\documentclass[11pt]{amsart}
\usepackage{amssymb,amsmath,mathrsfs}
\usepackage[ps2pdf,colorlinks=true,urlcolor=blue,
citecolor=red,linkcolor=blue,linktocpage,pdfpagelabels,bookmarksnumbered,bookmarksopen]{hyperref}
\usepackage{epsfig,graphicx,color}
\usepackage[english]{babel}

\usepackage[left=2.5cm,right=2.5cm,top=2.7cm,bottom=2.7cm]{geometry}

\newcommand{\N}{{\mathbb N}}

\newcommand{\C}{{\mathbb C}}
\newcommand{\R}{{\mathbb R}}
\newcommand{\iu}{{\rm i}}
\newcommand{\eps}{\varepsilon}
\renewcommand{\H}{{\mathbb H}}
\newcommand{\im}{{\Im}}
\newcommand{\re}{{\Re}}

\numberwithin{equation}{section}
\newtheorem{theorem}{Theorem}[section]
\newtheorem{proposition}[theorem]{Proposition}

\newtheorem{property}[theorem]{Property}
\newtheorem{lemma}[theorem]{Lemma}
\newtheorem{definition}[theorem]{Definition}
\theoremstyle{definition}
\newtheorem{remark}[theorem]{Remark}

\title[Point particle dynamics for general NLS problems]{Soliton dynamics
for a general \\ class of Schr\"odinger equations}

\author[R.\ Servadei]{Raffaella Servadei$^*$}
\author[M.\ Squassina]{Marco Squassina$^\dagger$ }

\thanks{$^*$Department of Mathematics,
University of Calabria,
Ponte Pietro Bucci 31B \newline
I-87036 Arcavacata di Rende,
Cosenza, Italy.
E-mail: {\em servadei@mat.unical.it}}

\thanks{$^\dagger$Department of Computer Science,
University of Verona,
C\`a Vignal 2, Strada Le Grazie 15 \newline
I-37134 Verona, Italy.
E-mail: {\em marco.squassina@univr.it}}

\address{Department of Mathematics
\newline\indent
University of Calabria
\newline\indent
Ponte Pietro Bucci 31B,
I-87036 Arcavacata di Rende,
Cosenza, Italy}
\email{servadei@mat.unical.it}

\address{Department of Computer Science
\newline\indent
University of Verona
\newline\indent
C\'a Vignal 2, Strada Le Grazie 15, I-37134 Verona, Italy}
\email{marco.squassina@univr.it}

\thanks{Both authors were supported by the 2007 MIUR national research
project entitled:\newline {\em ``Variational and Topological Methods in the Study of
Nonlinear Phenomena''}}
\subjclass[2000]{83C50, 81Q05, 35Q40, 35Q51, 35Q55, 37K40, 37K45}
\keywords{Nonlinear Schr\"odinger equations, magnetic fields, soliton dynamics}

\begin{document}

\begin{abstract}
The soliton dynamics  for
a general class of nonlinear focusing Schr\"odinger problems in presence of non-constant
external (local and nonlocal) potentials is studied by taking as initial datum the ground state
solution of an associated autonomous elliptic equation.
\end{abstract}
\maketitle

\section{Introduction and main result}\label{sec:intro}

\subsection{Introduction}
The aim of this paper is to study a general class of scalar and vectorial Schr\"odinger equations
in presence of local and nonlocal potentials, modelling an electric and magnetic field
and a Newtonian type interaction, respectively. This class of problems includes various
physically meaningful particular cases, that will be individually described in details later in this section.
In fact, we would also like to discuss the latest developments available in literature for this kind of issue, particularly when
approached via the technique initiated by the 2000 work of R.~Jerrard and J.~Bronski~\cite{bronski}.
More precisely, let $m\geq 1$, $N\geq 1$, $0<p<2/N$, $\eps>0$ and let
\begin{equation}\label{firo}
V:\R^N\to\R,
\qquad
A:\R^N\to\R^N,
\qquad
\Phi:\R^N\to\R,
\end{equation}
be $C^3(\R^N)$ functions satisfying suitable assumptions that will be
stated in the following. Then, if $\iu$ denotes the complex imaginary unit,
consider the Schr\"odinger equation
\begin{equation}
    \label{problematilde}
    \tag{$S$}
    \begin{cases}
        -\iu \eps \partial_t\zeta_\eps^j+ L_A\zeta_\eps^j+V(x)\zeta_\eps^j=|\zeta_\eps|^{2p}_j \zeta_\eps^j
        +\frac{1}{\eps^N}\Phi*|\zeta_\eps|^2_j\,\zeta_\eps^j & \text{in $\R^N\times (0,\infty)$}, \\
        \noalign{\vskip5pt}
        \zeta_\eps^j(x,0)=\zeta_0^j(x) & \text{in $\R^N$}, \\
        \noalign{\vskip3pt}
        j=1,\dots,m,
    \end{cases}
\end{equation}
where $\zeta_\eps=(\zeta^1_\eps,\dots,\zeta^m_\eps):\R^N\times \R^+\to\C^m$ is the unknown, the magnetic operator $L_A$
is defined as
\begin{equation*}
L_A\zeta:=-\frac{\eps^2}{2}\Delta\zeta-\frac{\eps}{\iu}A(x)\cdot\nabla\zeta
+\frac{1}{2}|A(x)|^2\zeta-\frac{\eps}{2\iu}{\rm div}_xA(x)\zeta,
\end{equation*}
the convolution is denoted by $(\Phi*v)(x):=\int\Phi(x-y)v(y)dy$, and
$$
|\zeta|_j^{2p}:=\alpha_j|\zeta^j|^{2p}
+\sum_{i\neq j}^m \gamma_{ij}|\zeta^i|^{p+1}|\zeta^j|^{p-1},\qquad
|\zeta|^2_j:=\beta_j|\zeta^j|^2+\sum_{i\neq j}^m\omega_{ij}|\zeta^i|^2,
$$
for some nonnegative constants $\alpha_i,\beta_i,\gamma_{ij},\omega_{ij}$ such that
$\gamma_{ij}=\gamma_{ji}$ and $\omega_{ij}=\omega_{ji}$, for all $i,j=1,\dots,m$.
By rescaling problem~\eqref{problematilde} with $\phi_\eps(x,t)=\zeta_\eps(\eps x, \eps t)$,
we reach the following system, where $\eps$ appears
now only in the arguments of the potentials $V$, $A$ and $\Phi$
\begin{equation}
    \label{problema}
    \tag{$P$}
    \begin{cases}
        -\iu \partial_t\phi_\eps^j+L_A\phi_\eps^j+V(\eps x)\phi_\eps^j=|\phi_\eps|^{2p}_j \phi_\eps^j
        +\Phi(\eps x)*|\phi_\eps|^2_j\,\phi_\eps^j & \text{in $\R^N\times (0,\infty)$}, \\
        \noalign{\vskip5pt}
        \phi_\eps^j(x,0)=\phi_0^j(x) & \text{in $\R^N$}, \\
        \noalign{\vskip3pt}
        j=1,\dots,m,
    \end{cases}
\end{equation}
with $\phi_\eps=(\phi^1_\eps,\dots,\phi^m_\eps):\R^N\times
\R^+\to\C^m$ and
\begin{equation}
\label{MagnOper}
L_A\phi:=-\frac{1}{2}\Delta\phi-\frac{1}{\iu}A(\eps
x)\cdot\nabla\phi +\frac{1}{2}|A(\eps x)|^2\phi-\frac{1}{2\iu}{\rm
div}_xA(\eps x)\phi.
\end{equation}
As we have already recalled, here $V:\R^N\to\R$ and $A:\R^N\to\R^N$ are an {\em electric} and
{\em magnetic} potentials respectively. The magnetic field $B$ is
$B=\nabla\times A$ in $\R^3$ and can be thought (and identified)
in general dimension as a $2$-form ${\mathbb H}^B$ of coefficients
$(\partial_i A_j-\partial_j A_i)$. We will keep
using the notation $B=\nabla\times A$ in any dimension $N$.

\vskip2pt
\noindent
We point out that the general Schr\"odinger problem~\eqref{problematilde} we aim to investigate contains,
as particular cases, the following physically meaningful situations.
\vskip8pt
\noindent
{\bf Class I.}
If $m=1$, $A=0$, $\beta_j=\omega_{ij}=\gamma_{ij}=0$ and $\alpha_j=1$, one finds:
    \begin{equation*}
    \begin{cases}
	\iu\eps\partial_t\zeta_\eps+\frac{\eps^2}{2}\Delta\zeta_\eps-V(x)\zeta_\eps
        +|\zeta_\eps|^{2p}\zeta_\eps=0 & \text{in $\R^N\times(0,\infty)$,} \\
        \noalign{\vskip3pt}
        \zeta_\eps(x,0)=\zeta_0(x) & \text{in $\R^N$}.
    \end{cases}
\end{equation*}
This is the classical {\em Schr\"odinger equation} with a spatial potential. For general results about local and global
existence of solutions, regularity, orbital stability and instability, we refer the reader
to~\cite{cazenave} and to the references therein. From the point of view of the semi-classical analysis
of standing wave solutions $\zeta_\eps(x,t)=u_\eps(x)e^{-\iu Et}$ for $E\in\R$, the Schr\"odinger equation reduces to
a semi-linear elliptic equation. In the last few years a huge literature has developed starting from
the celebrated paper by Floer and Weinstein~\cite{FW} (see
the monograph~\cite{ambook} by Ambrosetti and Malchiodi and references therein).
Concerning the soliton (or, equivalently, point-particle) dynamics, that is
the study of the qualitative behaviour of the solutions of this equation by choosing as initial datum a suitably rescaled
ground state solution of an associated elliptic problem, we refer e.g.~to the works~\cite{bronski,FGJS,GSS,Keerani2}
and to the recent monograph~\cite{carlesb} (see also e.g.~\cite{KN,KM}
for works in the mathematical physics community). Very recently, in~\cite{BeghMi},
Benci, Ghimenti and Micheletti provided the first result on the soliton dynamics
with uniform global estimates in time.
\vskip4pt
\noindent
{\bf Class II.}  If $m=1$, $\beta_j=\omega_{ij}=\gamma_{ij}=0$ and $\alpha_j=1$, one finds:
    \begin{equation*}
    \begin{cases}       \iu\eps\partial_t\zeta_\eps-\frac{1}{2}\big(\frac{\eps}{\iu}\nabla-A(x)\big)^2\zeta_\eps-V(x)\zeta_\eps
        +|\zeta_\eps|^{2p}\zeta_\eps=0 & \text{in $\R^N\times(0,\infty)$,} \\
        \noalign{\vskip3pt}
        \zeta_\eps(x,0)=\zeta_0(x) & \text{in $\R^N$}.
    \end{cases}
\end{equation*}
This is the {\em Schr\"odinger equation with a time-independent external magnetic field}. For general
facts about this equation, we refer again to~\cite{cazenave} and to the references therein.
For the semi-classical analysis of standing wave solutions, we refer the reader to the recent work~\cite{cjs}
and to the various references included. For the full (soliton) dynamics, we refer to the recent papers~\cite{selvit,squa}
which, to our knowledge, are the first contributions for this equation. In~\cite{squa}, the concentration center
is precisely the one predicted by the WKB theory.
\vskip4pt
\noindent
{\bf Class III.}
If $m=1$, $A=0$ and $\alpha_j=\gamma_{ij}=\omega_{ij}=0$, one finds:
    \begin{equation*}
    \begin{cases}       \iu\eps\partial_t\zeta_\eps+\frac{\eps^2}{2}\Delta\zeta_\eps-V(x)\zeta_\eps
        +\frac{\beta}{\eps^N}\Phi*|\zeta_\eps|^{2}\zeta_\eps=0 & \text{in $\R^N\times(0,\infty)$,} \\
        \noalign{\vskip3pt}
        \zeta_\eps(x,0)=\zeta_0(x) & \text{in $\R^N$}.
    \end{cases}
\end{equation*}
This is the {\em Hartree or Newton-Schr\"odinger type equation}. For basic facts about this equation, we
refer again to~\cite{cazenave} and references therein. For the study of standing waves in the semi-classical
regime, we refer to~\cite{WiWe} and the references included. The physical motivations for these equations
were detected by Penrose who derived the Schr\"odinger-Newton equation
by coupling the linear 3D Schr\"odinger equation with the Newton law of gravitation, yielding
\begin{equation*}
    \begin{cases}
\iu \eps \partial_t\zeta_\eps+\frac{\eps^2}{2}\Delta\zeta_\eps-V(x)\zeta_\eps+\Psi_\eps*\zeta_\eps=0
& \text{in $\R^3\times (0,\infty)$},\\
\noalign{\vskip4pt}
-\eps^2\Delta \Psi_\eps=\mu |\zeta_\eps|^2    & \text{in $\R^3$},
\end{cases}
\end{equation*}
where $\mu$ is a positive constant.
Of course, this system is equivalent to the nonlocal equation
\begin{equation*}
\iu \eps \partial_t\zeta_\eps+\frac{\eps^2}{2}\Delta\zeta_\eps-V(x)\zeta_\eps+
\Psi_\eps*|\zeta_\eps|^2\zeta_\eps=0\quad\text{in $\R^3\times (0,\infty)$},
\qquad
\Psi_\eps(x)=\frac{\mu}{4\pi\eps^{2}}\frac{1}{|x|}.
\end{equation*}
For the study of point-particle dynamics for this equation with
smooth nonlocal potentials, we refer the reader to~\cite{yau}, where the authors follow an
approach different from that used in~\cite{bronski,Keerani2}.
\vskip4pt
\noindent
{\bf Class IV.}
If $m=1$ and $\alpha_j=\gamma_{ij}=\omega_{ij}=0$, one finds:
    \begin{equation*}
    \begin{cases}       \iu\eps\partial_t\zeta_\eps-\frac{1}{2}\big(\frac{\eps}{\iu}\nabla-A(x)\big)^2\zeta_\eps-V(x)\zeta_\eps
        +\frac{\beta}{\eps^N}\Phi*|\zeta_\eps|^{2}\zeta_\eps=0 & \text{in $\R^N\times(0,\infty)$,} \\
        \noalign{\vskip3pt}
        \zeta_\eps(x,0)=\zeta_0(x) & \text{in $\R^N$}.
    \end{cases}
\end{equation*}
This is the {\em Hartree type equation with magnetic field}. As for the previous cases,
concerning the basic facts about this equation, we refer to~\cite{cazenave}. With respect
to the semiclassical analysis of standing waves we are not aware of any paper. The soliton dynamics
behaviour is contained in the present paper for smooth potentials.
\vskip4pt
\noindent
{\bf Class V.} If $m=2$, $A=0$ and $\beta_j=\omega_{ij}=0$, one finds:
\begin{equation*}
    \begin{cases}
        \iu\eps\partial_t\zeta_\eps^1+\frac{\eps^2}{2}\Delta\zeta_\eps^1-V(x)\zeta_\eps^1
        +\alpha_1|\zeta_\eps^1|^{2p}\zeta_\eps^1+\gamma_{12}|\zeta_\eps^2|^{p+1}|\zeta_\eps^1|^{p-1}=0 &
        \text{in $\R^N\times(0,\infty)$,} \\
        \iu\eps\partial_t\zeta_\eps^2+\frac{\eps^2}{2}\Delta\zeta_\eps^2-V(x)\zeta_\eps^2
        +\alpha_2|\zeta_\eps^2|^{2p}\zeta_\eps^2+\gamma_{12}|\zeta_\eps^1|^{p+1}|\zeta_\eps^2|^{p-1}=0 &
        \text{in $\R^N\times(0,\infty)$,} \\
        \noalign{\vskip3pt}
        \zeta_\eps(x,0)=\zeta_0(x) & \text{in $\R^N$}.
    \end{cases}
\end{equation*}
This is the {\em weakly coupled Schr\"odinger system with two components}. With respect
to the semiclassical analysis of standing waves, in the last few years the interest for
this systems has considerably increased. We refer for instance to~\cite{aCol,lw,mps,sirakov}
for the study of the structure of the associated ground states solutions (vector versus scalar
ground states depending upon the strength of the interaction $\gamma_{12}>0$). For the behaviour in the semiclassical limit,
we refer the reader to~\cite{dw,mps}. The soliton dynamics behaviour is contained in~\cite{squamont,mmp3},
essentially in the 1D case.

\subsection{The main result}
\label{preliminari}

In this section we shall provide the suitable background
allowing us to formulate the statement of the main theorem of the paper.

\subsubsection{Framework and main ingredients}
Throughout this paper we denote by $H_{A,\eps}$ the Hilbert space
defined as the closure of $C^\infty_c(\R^N;\C^m)$ under the scalar
product
$$
(u,v)_{H_{A,\eps}}=\Re\int(D u\cdot \overline{D
v}+V(\eps x)u\overline{v})dx\,,
$$
where $Du=(D_1u,\dots,D_Nu)$ and $D_j=\iu^{-1}\partial_j-A_j(\eps x)$, with induced norm
$$
\|u\|_{H_{A,\eps}}^2=\int\Big|\frac{1}{\iu}\nabla
u-A(\eps x)u\Big|^2dx+\int V(\eps x)|u|^2dx<\infty.
$$
The dual space of $H_{A,\eps}$ is denoted by $H_{A,\eps}'$,
while the space $H_{A,\eps}^2$ is the set of $u$ such that
$$
\|u\|_{H_{A,\eps}^2}^2=\|u\|_{L^2}^2+\Big\|\big(\frac{1}{\iu}\nabla-A(\eps x)\big)^2u\Big\|_{L^2}^2<\infty.
$$
Finally, $H^1(\R^N;\C^m)$ is equipped with
the standard norm $\|\phi\|_{H^1}^2=\|\nabla\phi\|_{L^2}^2+\|\phi\|_{L^2}^2$.
We study problem~\eqref{problema} for an initial datum $\phi_0:\R^N\to\C^m$ given by
\begin{equation} \label{initialD} \tag{$I$}
\phi_0^j(x)=r_j(x-x_\eps(0))
e^{\iu[A(\eps x_\eps(0))\cdot(x-x_\eps(0))+x\cdot\xi_\eps(0)]},\qquad j=1,\dots,m
\end{equation}
where $x_0/\eps$ and $\xi_0$ are the initial position and the initial
velocity in $\R^N$ of the following first order differential system
\begin{equation}
    \label{DriveS}
\tag{$D$}
    \begin{cases}
        \dot x_\eps(t)=\xi_\eps(t), & \\
        \noalign{\vskip3pt}
        \dot \xi_\eps(t)=-\eps\nabla V(\eps x_\eps(t))-\eps\xi_\eps(t)\times B(\eps x_\eps(t)), & \\
        \noalign{\vskip1pt}
        x_\eps(0)=\frac{x_0}{\eps}, &\\
        \noalign{\vskip1pt}
        \xi_\eps(0)=\xi_0\,, &
    \end{cases}
\end{equation}
with $B=\nabla \times A$. Notice that, for the solution of~\eqref{DriveS}, we have
\begin{equation}
    \label{conversion}
x_\eps(t)=\frac{x(\eps t)}{\eps},\quad \xi_\eps(t)=\xi(\eps t),\qquad
    \begin{cases}
        \dot x(t)=\xi(t), & \\
        \noalign{\vskip3pt}
        \dot \xi(t)=-\nabla V(x(t))-\xi(t)\times B(x(t)), & \\
        \noalign{\vskip1pt}
        x(0)=x_0, &\\
        \noalign{\vskip1pt}
        \xi(0)=\xi_0. &
    \end{cases}
\end{equation}
The rescaled components $(x(t),\xi(t))$ of system~\eqref{conversion} might appear
in the proofs of some result. Notice that the initial datum referred to the original
problem~\eqref{problematilde} reads as
\begin{equation*}
\zeta_0^j(x)=\phi_0^j\Big(\frac{x}{\eps}\Big)=r_j\Big(\frac{x-x_0}{\eps}\Big)
e^{\frac{\iu}{\eps}[A(x_0)\cdot(x-x_0)+x\cdot\xi_0]},\quad
x\in\R^N,\,\, j=1,\dots,m
\end{equation*}
This is the usual formula for the (soliton) initial datum considered in~\cite{bronski,Keerani2}
when $A=0$ and in~\cite{selvit,squa} when $A\neq 0$. Furthermore, we assume that $r=(r_1,\dots,r_m)\in
H^1(\R^N,\R^m)$ is (up to translation) a real ground state solution
of the elliptic system
\begin{equation}
    \label{seMF}
\tag{$S$}
    \begin{cases}
        -\frac{1}{2}\Delta r_j+r_j=|r|^{2p}_jr_j  & \quad\text{in $\R^N$}, \\
        \noalign{\vskip3pt}
       \,\, j=1,\dots,m,
    \end{cases}
\end{equation}
with respect to the notation of $|\cdot|_j$ previously introduced. We also set
\begin{equation}\label{valueofm}
m_j:=\|r_j\|_{L^2}^2,\qquad j=1,\dots,m,\,\,\qquad
M:=\sum_{j=1}^m m_j.
\end{equation}
Notice that, setting for all $t\in\R^+$
\begin{equation}
\label{Hamilt} {\mathcal H}(t)=\frac{1}{2}|\xi_\eps(t)|^2+V(\eps x_\eps(t))+{\mathcal M},
\end{equation}
where
$$
{\mathcal M}:=-\frac{\Phi(0)}{2M}\Big\{\sum_{j=1}^m\beta_j m_j^2+\sum_{i\neq j}^m\omega_{ij}m_im_j\Big\},
$$
it follows that ${\mathcal H}$ is a first integral associated with~\eqref{DriveS}, namely
$$
{\mathcal H}(t)={\mathcal H}(0)=\frac{1}{2}|\xi_0|^2+V(x_0)+{\mathcal M},\qquad\text{for all $t\in\R^+$}.
$$
In turn, the function ${\mathcal H}$ is independent of both time and $\eps>0$.

\subsubsection{Assumptions on the potentials}

We first give the following

\begin{definition}
    \label{defadmissibilityy}
Consider the potentials $V:\R^N\to\R$, $A:\R^N\to\R^N$ and $\Phi:\R^N\to\R$ and a ground state solution $r$
of~\eqref{seMF} which is chosen to build up the initial datum~\eqref{initialD}.
We say that $(V,A,\Phi,r)$ is an admissible string for
the point particle dynamics of problem~\eqref{problema} if $r_j$ is radially symmetric,
$x_ir_j\in L^2(\R^N)$ for all $i=1,\dots,N$ and $j=1,\dots,m$ and the following
Properties~\ref{wellP} (well-posedness) and~\ref{stabPhi}
(non-degeneracy/energy convexity inequality) hold true.
\end{definition}

\begin{property}[Well-posedness]
\label{wellP}
Assume that $0<p<2/N$.\ Then, for all $\eps>0$ and
$\phi_0\in\ H_{A,\eps}$, there exists a unique global solution
$$
\phi_\eps\in C(\R^+,H_{A,\eps})\cap C^1(\R^+,H_{A,\eps}'),
$$
of problem~\eqref{problema} with
$\sup\limits_{t\in\R^+}\|\phi_\eps(t)\|_{H_{A,\eps}}<\infty$.
Furthermore, the mass ${\mathcal N}_{\eps}^j$ associated with $\phi_\eps^j(t)$,
\begin{equation*}
{\mathcal N}_{\eps}^j(t):=\int |\phi_{\eps}^j(t)|^{2}dx, \qquad
t\in\R^+, \quad j=1,\dots, m,
\end{equation*}
and  the total energy $E_\eps$,
\begin{align*}
 E_{\eps}(t) &
   :=\frac{1}{2}\int\Big|\frac{1}{\iu}\nabla \phi_\eps(x)-A(\eps x)\phi_\eps\Big|^2dx
  +\int V(\eps x)|\phi_\eps(x)|^2dx
-\frac{1}{p+1}\sum_{j=1}^m\alpha_j\int |\phi_\eps^j(x)|^{2p+2}dx  \\
& -\frac{1}{p+1}\sum_{i,j,\,i\neq j}^m \gamma_{ij}\int |\phi_\eps^i(x)|^{p+1}|\phi_\eps^j(x)|^{p+1}dx \notag \\
& -\frac{1}{2}\sum_{j=1}^m\beta_j \iint \Phi(\eps(x-y))|\phi_\eps^j(x)|^2|\phi_\eps^j(y)|^2 dx dy \\
& -\frac{1}{2}\sum_{i,j,\,i\neq j}^m\omega_{ij} \iint
\Phi(\eps(x-y))|\phi_\eps^i(x)|^2|\phi_\eps^j(y)|^2 dx dy,  \qquad t\in\R^+,
\end{align*}
are conserved in time, namely
$$
{\mathcal N}_{\eps}^j(t)={\mathcal N}_{\eps}^j(0)\quad\text{and}\quad E_{\eps}(t)=
E_{\eps}(0),\qquad\text{ for all $t\in\R^+$,} \quad j=1,\dots,m.
$$
Finally if $\phi_0\in H_{A,\eps}^2$, then $\phi_\eps\in
C(\R^+,H_{A,\eps}^2)\cap C^1(\R^+,L^2(\R^N;\C^m))$.
\end{property}

\noindent
We also consider the functional ${\mathcal E}:H^1(\R^N;\R^m)\to\R$ associated with system~\eqref{seMF}
\begin{equation*}
{\mathcal E}(u) =\frac{1}{2}\int |\nabla u(x)|^2dx
-\sum_{j=1}^m\frac{\alpha_j}{p+1}\int |u_j(x)|^{2p+2}dx
 -\sum_{i,j,\,i\neq j}^m \frac{\gamma_{ij}}{p+1}\int |u_i(x)|^{p+1}|u_j(x)|^{p+1}dx.
\end{equation*}
In a large range of relevant situations, a ground state
solution $r$ of~\eqref{seMF} satisfies the characterization
\begin{equation}
\label{variatcaract-r} {\mathcal E}(r)=\min\{{\mathcal E}(u): u\in
H^1(\R^N,\R^m),\,\|u\|_{L^2}=\|r\|_{L^2}\}.
\end{equation}

\noindent
For $m=1$ this is a classical fact. For $m=2$ see e.g.~\cite{mmp2}.

\noindent We consider now the following

\begin{property}[Non-degeneracy/Energy convexity inequality]
    \label{stabPhi}
There exist two positive constants $C$ and $C'$ such that the following
condition holds: if $U\in
H^1(\R^N;\C^m)$ is such that $\|U\|_{L^2}=\|r\|_{L^2}$, where $r$
is a ground state solution of~\eqref{seMF}, then
\begin{equation}
\Gamma_U \leq C\left(\mathcal E(U)-\mathcal E(r)\right),
\end{equation}
where
\begin{equation}
    \label{gammaUdef}
\Gamma_{U}=\inf_{\overset{y\in\R^N}{
\theta_1, \dots, \theta_m\in [0,2\pi)}
}\|U(\cdot)-\big(e^{\iu \theta_1}r_1(\cdot +y), \dots,
e^{\iu \theta_m}r_m(\cdot +y)\big) \|_{H^1}^{2},
\end{equation}
provided that $\Gamma_U<C'$.
\end{property}
The energy convexity inequality is essentially a feature of a ground state
solution $r$. It is generally a quite delicate issue to consider, based upon
nontrivial spectral estimates and the fact that the kernel of the linearized operator is
$N$-dimensional and spanned by the partial derivatives $\partial_j r$ of $r$.
Let us point out which is the current
knowledge of particular cases, within our framework, where this assumption is indeed satisfied.
For the Schr\"odinger equation with or without magnetic field, Property~\ref{stabPhi}
is satisfied, since the (unique) ground state solution of $-\frac{1}{2}\Delta r+r=r^{2p+1}$ is
non-degenerate and satisfies suitable spectral estimates (see the striking works of Weinstein~\cite{We1,We2}).
For systems, already in the case of two components, the situation is still very far from being
completely understood. On the other hand, very recently Dancer and Wei have proved in~\cite{dw} the existence
of non-degenerate ground state solutions in some particular cases, providing an important tool in connection
with Property~\ref{stabPhi}.
In the one dimensional case, Property~\ref{stabPhi} has been verified in~\cite{mmp3}
for two-components weakly coupled nonlinear Schr\"odinger system. The main obstacle in dealing
with the higher dimensional case is the smoothness of the energy functional ${\mathcal E}$ which is not of
class $C^2$ due to the presence of the coupling terms $\int |\phi^i|^{p+1}|\phi^j|^{p+1}$, being $p<2/N<1$.

\smallskip
\subsubsection{Statement of the result}
On the external potentials $V$ and $A$,
on the nonlocal term $\Phi$ and on the ground state solution $r$
of~\eqref{seMF} which is chosen to build the initial datum~\eqref{initialD},
we assume that they are admissible for the point particle dynamics in the sense indicated above
and that the following conditions hold:
\vskip6pt
\noindent

\noindent {\bf (V)} $V\in C^3(\R^N)$ is positive and
$\|V\|_{C^3}<\infty$\,;\\

\noindent {\bf (A)} $A\in C^3(\R^N;\R^N)$ with $\|A\|_{C^3}<\infty$;\\

\noindent {\bf ($\boldsymbol{\Phi}$)} $\Phi\in C^3(\R^N)$ positive with $\|\Phi\|_{C^3}<\infty$.

\vskip8pt
\noindent
We shall think $\Phi$ as a smooth function decaying at infinity as $|x|^{-\rho}$ for some $\rho>0$
(for instance, in $\R^N$ with $N\geq 3$, decaying as the Coulomb potential $|x|^{2-N}$)
having a maximum point at the origin.
\vskip8pt
\noindent
Under the previous assumptions, we can state the main result of this paper.

\begin{theorem}\label{th:main}
    Assume that $\Phi=0$ in the vectorial case $m>1$.
    Let $\phi_{\eps}$ be the family of solutions to problem~\eqref{problema} corresponding to the initial datum~\eqref{initialD}
    modelled on a ground state $r$ of~\eqref{seMF} and let $(x_\eps(t), \xi_\eps(t))$ be the solution
    of~\eqref{DriveS}. Then there exist $\delta>0$, $\eps_0>0$ and shift functions $\theta^1_{\eps}, \dots,
\theta^m_{\eps} :\R^+\to[0,2\pi)$ such that, if $\|A\|_{C^2}<\delta$, then
\begin{equation*}
    \phi_\eps^j(x,t)=e^{\iu(\xi_\eps(t)\cdot x+\theta^j_\eps(t)+A(\eps x_\eps(t))
    \cdot (x-x_\eps(t))}r_j(x-x_\eps(t))+\omega^j_\eps(x,t),
\end{equation*}
where $\|\omega^j_\eps(t)\|_{H^1}\leq {\mathcal O}(\eps)$,
for all $\eps\in(0,\eps_{0})$ and $j=1,\dots, m$, locally uniformly in time with the time scale $\eps^{-1}$.
Furthermore, without restrictions on $\|A\|_{C^2}$, there exists $\eps_0>0$ such that
\begin{equation}
	\label{secondconcl-A}
	|\phi_\eps^j(x,t)|=r_j(x-x_\eps(t))+\hat\omega_\eps^j(x,t),
\end{equation}
where $\|\hat\omega_\eps^j\|_{H^1}\leq {\mathcal O}(\eps)$,
for all $\eps\in(0,\eps_{0})$ and $j=1,\dots, m$, locally uniformly in time
with the time scale $\eps^{-1}$.
\end{theorem}
\vskip1pt

This kind of results has the origin in some works in linear geometric asymptotics which go back to the 70's
(see~\cite{GS}). We stress that, in the vectorial case $m>1$, we are not aware of any physically reasonable model including the
nonlocal coupling terms. Hence, for $m>1$, we consider systems of coupled Schr\"odinger equations with local terms,
which are being extensively studied in the literature of recent years.

\begin{remark}
Rescaling back to problem~\eqref{problematilde}, the
approximated representation formula reads as
\begin{equation*}
    \zeta_\eps^j(x,t)=e^{\frac{\iu}{\eps}(\xi(t)\cdot x+\vartheta^j_\eps(t)+A(x(t))
    \cdot (x-x(t))}r_j\Big(\frac{x-x(t)}{\eps}\Big)+\Xi^j_\eps(x,t),
\end{equation*}
locally uniformly in time, where we have set $\vartheta^j_\eps(t)=\eps\theta^j_\eps(t/\eps)$ and
$\Xi^j_\eps(x,t)=\omega^j_\eps(x/\eps,t/\eps)$, which reads as in~\cite{squa} and in the
previously cited papers in the particular cases $m=1$, $A=0$ and $\Phi=0$.
\end{remark}

\vskip25pt
\begin{center}\textbf{Plan of the paper.}\end{center}
In Section~\ref{preliminarysection}, we prove various preliminary Lemmas, particularly
focused on the asymptotic behaviour of the energy, for $\eps$ small.
In Section~\ref{masssection}, we prove some Lemmas, focused on the asymptotic
behaviour of the density and of the momentum associated with the solution, for $\eps$ small.
In Section~\ref{error-estimate}, we prove a result yielding a precise control
on the norm of the error function $\omega^j_\eps$ which appears in Theorem~\ref{th:main}.
Finally, in Section~\ref{proof-section}, we conclude the proof of the main result, Theorem~\ref{th:main}.

\vskip20pt
\begin{center}\textbf{Notations.}\end{center}
\begin{enumerate}
\item The imaginary unit is denoted by $\iu$.
\item The conjugate of any $z\in\C$ is denoted by $\bar z$, the real and imaginary parts by $\Re z$ and $\Im z$.
\item The symbol $\R^+$ means the positive real line $[0,\infty)$.
\item The ordinary inner product between two vectors $a,b\in\R^N$ is denoted by $a \cdot b$.
\item The standard $L^p$ norm, $1<p\leq\infty$ of a function $u$ is denoted by $\|u\|_{L^p}$.
\item The symbols $\partial_t$ and $\partial_j$ mean $\frac{\partial}{\partial t}$ and $\frac{\partial}{\partial x_j}$ respectively. $\Delta$ means $\frac{\partial^2}{\partial x_{1}^2}+\cdots+\frac{\partial^2}{\partial x_{N}^2}$.
\item The symbol $C^k(\R^N;\C^m)$, for $k\in\N$, denotes the space of functions with continuous derivatives up to
the order $k$. Sometimes $C^k(\R^N;\C^m)$ is endowed with the norm
$$
\|\phi\|_{C^k}=\sum_{|\alpha|\leq k}\|D^\alpha\phi\|_{L^\infty}<\infty.
$$
\item The symbol $\int f(x)dx$ stands for the integral of $f$ over $\R^N$ with the Lebesgue measure.
\item The symbol $C^{2*}$ denotes the dual space of $C^2$. The norm of a $\nu$ in $C^{2*}$ is
$$
\|\nu\|_{C^{2*}}=\sup\Big\{\big|\int \phi(\eps x)\nu dx\big|:\,\phi\in C^2(\R^N),\,\,\|\phi\|_{C^2}\leq 1\Big\}.
$$
Clearly, $C^{2*}$ contains the space of bounded Radon measures.
\item $C$ denotes a generic positive constant, which may vary inside a chain of inequalities.
\item ${\mathcal O}(\eps)$ is a generic function such that
the $\limsup$ of $\eps^{-1}{\mathcal O}(\eps)$ is finite, as $\eps\to 0$.
\end{enumerate}

\medskip
\medskip

\section{Some preliminary stuff}
\label{preliminarysection}
\noindent

Observe that, from Property~\ref{wellP}, due to the choice of
the initial datum~\eqref{initialD}, the masses ${\mathcal
N}_{\eps}^j(t)$ are also {\em independent} of $\eps$. Indeed, via the
mass conservation law, by the form of the initial datum and~\eqref{valueofm},
we have
\begin{equation}\label{eqmi}
{\mathcal N_{\eps}^j(t)}={\mathcal
N_{\eps}^j(0)}=\int|\phi_{\eps}^j(x,0)|^{2}dx=\int
\Big|r_j\Big(x-x_\eps(0)\Big)
\Big|^{2}dx=\|r_j\|^2_{L^2}=m_j,
\end{equation} for all $\eps>0$, $t\in\R^+$
and $j=1,\dots,m$.

\vskip4pt
\noindent
We now recall a useful identity (see e.g.~\cite[Lemma 3.3]{Keerani2}).

\begin{lemma}\label{pote}
Assume that $g:\R^N\to\R$ is a function of class $C^{2}(\R^N)$,
$\|g\|_{C^2}<\infty$, and that $r$ is a ground state solution
of~\eqref{seMF}. Then, as $\eps$ goes to zero, for any $i=1,\dots,m$ it holds
\begin{equation*}
\int g(\eps x+y)r^2_i(x)dx=\int g(y)r^2_i(x)dx+{\mathcal O}(\eps^2),
\end{equation*}
for every $y\in \R^N$.
\end{lemma}

\vskip2pt
\noindent
In a similar fashion, we have the following counterpart to be used for the nonlocal term.

\begin{lemma}\label{potePhi}
Assume that $g:\R^N\to\R$ is a function of class $C^{2}(\R^N)$,
$\|g\|_{C^2}<\infty$, and that $r$ is a ground state solution
of~\eqref{seMF}. Then, as $\eps$ goes to zero, for any $i,j=1,\dots,m$ it holds
\begin{equation*}
\iint g(\eps (x-y))r^2_i(x)r_j^2(y)dxdy=m_im_j g(0)+{\mathcal O}(\eps^2).
\end{equation*}
\end{lemma}
\begin{proof}
By Taylor expansion, for some point $\xi$ of the form $\xi=\eps\tau (x-y)$ with $\tau\in (0,1)$, we have
\begin{align*}
&\iint g(\eps (x-y))r^2_i(x)r_j^2(y)dxdy=  \\
&\qquad  =g(0) \iint r^2_i(x)r_j^2(y)dxdy
+\eps\sum_{h=1}^ND_hg(0) \cdot \iint (x_h-y_h) r^2_i(x)r_j^2(y)dxdy \\
&\qquad +\frac{\eps^2}{2} \sum_{h,k=1}^N \iint D^2_{hk}g(\xi) (x_h-y_h)(x_k-y_k)   r^2_i(x)r_j^2(y)dxdy \\
&\qquad  = m_i m_jg(0) +\eps\sum_{h=1}^ND_hg(0) \int x_h r^2_i(x)dx \int r_j^2(y)dy \\
&\qquad -\eps\sum_{h=1}^ND_hg(0) \int y_h r^2_j(y)dy \int r_i^2(x)dx  \\
&\qquad +\frac{\eps^2}{2} \sum_{h,k=1}^N \iint D^2_{hk}g(\xi) (x_h-y_h)(x_k-y_k)  r^2_i(x)r_j^2(y)dxdy \\
&\qquad =  m_i m_jg(0) +\frac{\eps^2}{2} \sum_{h,k=1}^N \iint D^2_{hk}g(\xi)(x_h-y_h)(x_k-y_k)  r^2_i(x)r_j^2(y)dxdy \\
&\qquad =  m_i m_jg(0) +{\mathcal O}(\eps^2).
\end{align*}
In the above computations we used the fact that $|D^2_{hk}g(\xi)|\leq \|g\|_{C^2}<\infty$,
that, since $r_i$ is radially symmetric, $\int z_h r^2_i(z)dz=0$ and, finally, that
$z_h r_i\in L^2(\R^N)$ for any $h$ and $i$ (cf.~Definition~\ref{defadmissibilityy}).
\end{proof}

\vskip4pt
\noindent
In the next result we obtain an asymptotic formula for the energy, linking the functionals
$E_\eps$, ${\mathcal E}$ and ${\mathcal H}$, up to an error ${\mathcal O}(\eps^2)$ (see also~\cite{squa}).

\begin{lemma}
    \label{estUNO}
For every $t\in\R^+$, as $\eps$ goes to zero, it holds
$$
E_\eps(t)={\mathcal E}(r)+M{\mathcal H}(t)+{\mathcal O}(\eps^2).
$$
\end{lemma}
\begin{proof}
Taking into account that, in view of Lemma~\ref{pote}, for all
$j=1,\dots, m$ we have
\begin{align*}
& \int r^2_j(x)|A(\eps x+x_0)|^2dx=|A(x_0)|^2m_j+{\mathcal O}(\eps^2),   \\
& \int r^2_j(x)A(\eps x+x_0)\cdot (A(x_0)+\xi_0)dx=A(x_0)\cdot
(A(x_0)+\xi_0)m_j+{\mathcal O}(\eps^2),
\end{align*}
as $\eps$ goes to zero, it is readily checked that, for any
$j=1,\dots,m$, we get
\begin{equation*}
 \int\Big|\Big(\frac{\nabla}{\iu}-A(\eps x)\Big)
\Big(r_j\Big(x-x_\eps(0)\Big)e^{\iu[A(x_0)\cdot(x-x_\eps(0))+x\cdot\xi_0]}\Big)\Big|^2dx
=\int\left|\nabla r_j(x)\right|^2dx +m_j|\xi_0|^2+{\mathcal
O}(\eps^2).
\end{equation*}
In turn, by combining the conservation of energy (see
Property~\ref{wellP}) and the conservation of the function
${\mathcal H}$ (see definition~\eqref{Hamilt}), taking into account
Lemma~\ref{pote} and Lemma~\ref{potePhi}, as $\eps$ goes to zero, we get
$$\begin{aligned}
E_\eps(t)&=E_\eps(0)
=E_\eps\Big(r(x-x_\eps(0))e^{\iu[A(x_0)\cdot(x-x_\eps(0))
+x\cdot\xi_0]}\Big) \\
&=\frac{1}{2}\sum_{j=1}^m\int\Big|\Big(\frac{\nabla}{\iu}-A(\eps
x)\Big)
\Big(r_j\Big(x-x_\eps(0)\Big)e^{\iu[A(x_0)\cdot(x-x_\eps(0))+x\cdot\xi_0]}\Big)\Big|^2dx \\
& \quad +\sum_{j=1}^m\int  V(x_0+\eps x)r^{2}_j(x)dx
-\sum_{j=1}^m\frac{\alpha_j}{p+1}\int |r_j|^{2p+2}dx
-\!\!\!\sum_{i,j=1,\,i\neq j}^m\frac{\gamma_{ij}}{p+1}\int |r_i|^{p+1}|r_j|^{p+1}dx \\
& \quad -\sum_{j=1}^m\frac{\beta_j}{2} \iint \Phi(\eps(x-y))|r_j(x)|^2|r_j(y)|^2
dx dy
 -\!\!\sum_{i,j,\,i\neq j}^m\frac{\omega_{ij}}{2} \iint \Phi(\eps(x-y))|r_i(x)|^2|r_j(y)|^2 dx dy \\
&={\mathcal E}(r)+\sum_{j=1}^m \int  V(x_0+\eps x)r^{2}_j(x)dx
+\frac{1}{2}\sum_{j=1}^m m_j|\xi_0|^2+M{\mathcal M}+{\mathcal O}(\eps^2) \\
& ={\mathcal E}(r)+\sum_{j=1}^m m_jV(x_0)+\frac{1}{2}\sum_{j=1}^m m_j|\xi_0|^2+M{\mathcal M}+{\mathcal O}(\eps^2)
 ={\mathcal E}(r)+M{\mathcal H}(t)+{\mathcal O}(\eps^2).
\end{aligned}$$
\end{proof}

\noindent
The function $p^A_\eps:\R^N\times\R^+\to\R^{m+N}$ is the
(magnetic) momentum of $\phi_\eps$, defined as
\begin{equation}
    \label{paepsilon}
p^A_\eps(x,t):=\im\big(\bar\phi_\eps(x,t)(
\nabla\phi_\eps(x,t)-\iu A(\eps x)\phi_\eps(x,t))\big),\quad
x\in\R^N,\, t\in\R^+.
\end{equation}
Then, we have the following

\begin{lemma}
    \label{normBB}
Let $\phi_\eps$ be the solution to problem~\eqref{problema}
corresponding to the initial datum~\eqref{initialD}. Then there
exists a positive constant $C$ such that
$$
\big\|\iu^{-1}\nabla\phi_\eps(\cdot,t)-A(\eps
x)\phi_{\eps}(\cdot,t)\big\|^{2}_{L^2}\leq C,
$$
for all $t\in\R^+$ and any $\eps\in (0,1]$. In particular,
$$
\sup\limits_{t\in\R^+}\big|\int
p^A_\eps(x,t)dx\big|<\infty.
$$
\end{lemma}

\begin{proof}
By Property~\ref{wellP} the total energy $E_\eps$ is conserved
and, in addition, can be bounded independently of $\eps$ (due to
the choice of initial datum, see Lemma~\ref{estUNO}). Taking into
account the positivity of $V$ and the definition of $E_\eps$, it
follows that there exists a positive constant $C$ such that
\begin{equation}\label{norma22}
\begin{aligned}
& \Big\|\frac{1}{\iu}\nabla\phi_\eps(\cdot,t)-A(\eps
x)\phi_{\eps}(\cdot,t)\Big\|^{2}_{L^2}  =\int \Big|
\frac{1}{\iu}\nabla\phi_\eps(x,t)-A(\eps
x)\phi_{\eps}(x,t)\Big|^2 dx\\
& = 2E_\eps(t)-2\int V(\eps x)|\phi_{\eps}(x,t)|^{2}dx
+\frac{2}{p+1}\sum_{j=1}^m\alpha_j\int |\phi^j_\eps(x,t)|^{2p+2}dx \\
& \qquad +\frac{2}{p+1}\sum_{i,j,\,i\neq j}^m \gamma_{ij}\int |\phi^i_\eps(x,t)|^{p+1}|\phi^j_\eps(x,t)|^{p+1}dx \\
& \qquad +\sum_{j=1}^m\beta_j \iint \Phi(\eps(x-y))|\phi^j_\eps(x,t)|^2|\phi^j_\eps(y,t)|^2 dx dy \\
& \qquad +\sum_{i,j,\,i\neq j}^m\omega_{ij} \iint \Phi(\eps(x-y))|\phi^i_\eps(x,t)|^2|\phi^j_\eps(y,t)|^2 dx dy \\
& \leq C+\frac{2}{p+1}\sum_{j=1}^m\alpha_j\int
|\phi_{\eps}^j(x,t)|^{2p+2}dx
 +\frac{2}{p+1}\sum_{i,j,\,i\neq j}^m \gamma_{ij}\int |\phi_{\eps}^i(x,t)|^{p+1}|\phi_{\eps}^j(x,t)|^{p+1}dx \\
& \qquad +\sum_{j=1}^m\beta_j \iint \Phi(\eps(x-y))|\phi_{\eps}^j(x)|^2|\phi_{\eps}^j(y)|^2 dx dy \\
& \qquad +\sum_{i,j,\,i\neq j}^m\omega_{ij} \iint
\Phi(\eps(x-y))|\phi_{\eps}^i(x,t)|^2|\phi_{\eps}^j(y,t)|^2 dx dy.
\end{aligned}
\end{equation}
By combining the diamagnetic inequality (see e.g.~\cite{EL} for a proof)
$$
|\nabla |\phi_\eps^j||\leq \Big|\Big(\frac{\nabla}{\iu}-A(\eps
x)\Big)\phi_\eps^j\Big|,\qquad \text{a.e.\ in $\R^N$}
$$
with the Gagliardo-Nirenberg inequality, setting
$\vartheta=\frac{pN}{2p+2}\in (0,1)$, we obtain
$$
\|\phi_\eps^j(\cdot,t)\|_{L^{2p+2}}\leq
\|\phi_\eps^j(\cdot,t)\|_{L^{2}}^{1-\vartheta} \|\nabla
|\phi_\eps^j(\cdot,t)|\|^{\vartheta}_{L^{2}}\leq
\|\phi_\eps^j(\cdot,t)\|_{L^{2}}^{1-\vartheta}
\Big\|\Big(\frac{\nabla}{\iu}-A(\eps
x)\Big)\phi_\eps^j(\cdot,t)\Big\|^{\vartheta}_{L^{2}}
$$
for any $j=1,\dots,m$. While, by the conservation of mass, we
deduce that $$\|\phi_\eps^j(\cdot,t)\|_{L^2}^2 ={\mathcal
N}_\eps^j(t)=m_j\qquad j=1,\dots,m,$$ independently of $\eps$ (see
formula~\eqref{eqmi}). Hence, for all $\eps>0$, we get
\begin{equation}\label{diamagnetica}
\|\phi_\eps^j(\cdot,t)\|_{L^{2p+2}}^{2p+2} \leq
m_j^{(1-\theta)(p+1)}\Big\|\frac{1}{\iu}\nabla\phi_\eps^j(\cdot,t)-A(\eps
x)\phi_\eps^j(\cdot,t)\Big\|^{pN}_{L^{2}}\leq
C(\Upsilon_\eps(t))^{pN},
\end{equation}
for any $j=1,\dots,m$ and for some positive constant $C$, where we
have set, for $t>0$,
$$
\Upsilon_\eps(t)=\max_{j=1,\dots,m}\Upsilon^j_\eps(t),\qquad
\Upsilon^j_\eps(t)=\big\|\frac{1}{\iu}\nabla\phi_\eps^j(\cdot,t)-A(\eps
x)\phi_\eps^j(\cdot,t)\big\|_{L^{2}}.
$$
Observe also that, as $\Phi$ is uniformly bounded, for any $i,j=1,\dots,m$ we have
\begin{equation*}
 \iint \Phi(\eps(x-y))|\phi_\eps^i(x,t)|^2|\phi_\eps^j(y,t)|^2 dxdy
 \leq C \int |\phi_\eps^i(x,t)|^2dx \int |\phi_\eps^j(y,t)|^2 dy=C m_im_j.
\end{equation*}
Finally, notice that, by Young inequality
\begin{equation}
    \label{youngban}
\int |\phi_\eps^i(x,t)|^{p+1}|\phi_\eps^j(x,t)|^{p+1}dx\leq
\frac{1}{2}\|\phi_\eps^i(\cdot,t)\|^{2p+2}_{L^{2p+2}}
+\frac{1}{2}\|\phi_\eps^j(\cdot,t)\|^{2p+2}_{L^{2p+2}},
\end{equation}
for any $j=1,\dots,m$.
Putting now together all the previous inequalities from
\eqref{norma22} to~\eqref{youngban},
we finally obtain $(\Upsilon_\eps(t))^2\leq C+C(\Upsilon_\eps(t))^{pN}$
for $t>0$.
Taking into account that $pN<2$ by the assumption on $p$, if
$\Upsilon_\eps(t)$ was unbounded with respect to $t$ or $\eps$, the above
inequality would yield a contradiction. Hence $\Upsilon_\eps$ is
uniformly bounded with respect to $t$ and $\eps$, so that
the first assertion of Lemma~\ref{normBB} holds.
In order to prove the final assertion observe that, taking into
account the mass conservation law, by
H\"older inequality we get
\begin{equation*}
\left|\int p^A_\eps(x,t)dx\right| \leq \int |p^A_\eps(x,t)|dx\leq
\|\phi_\eps(\cdot,t)\|_{L^2}\Big\|\frac{1}{\iu}
\nabla\phi_\eps(\cdot,t)-A(\eps
x)\phi_\eps(\cdot,t)\Big\|_{L^2}\leq C,
\end{equation*}
for all $t\in\R^+$. The assertion follows by taking the supremum
over $t$ in $\R^+$.
\end{proof}

\noindent
For the next lemma we need to introduce the  total magnetic momentum $q^A_\eps$ defined as
$$
q^A_\eps(x,t)=\sum_{j=1}^m (p_\eps^A)^j(x,t), \qquad x\in\R^N, \quad t>0.
$$
Then, on a suitable function $\psi_\eps$ (related to the solution $\phi_\eps$), we have the following
\begin{lemma}
    \label{estDUE}
Let $\phi_{\eps}$ be the family of solutions to problem~\eqref{problema}
corresponding to the initial datum~\eqref{initialD}.
Let us set, for any $\eps>0$, $t\in\R^+$ and $x\in\R^N$
\begin{equation}
\label{veps} \psi_{\eps}^j(x,t) =e^{-{\rm\iu}\xi_\eps(t)\cdot [x+x_\eps(t)]} e^{-\iu A(\eps
x_\eps(t))\cdot x} \,\phi_{\eps}^j\left(x+x_\eps(t), t\right),\qquad j=1,\dots,m,
\end{equation}
where $(x_\eps(t), \xi_\eps(t))$ is the solution of system~\eqref{DriveS}.
Then, as $\eps$ goes to zero,
\begin{align*}
\mathcal E(\psi_{\eps}(t)) -{\mathcal E}(r) &=M{\mathcal H}(t)
-\int V(\eps x)|\phi_{\eps}(x,t)|^{2}dx
+\frac12 M|\xi_\eps(t)+A(\eps x_\eps(t))|^2 \\
& -(\xi_\eps(t)+A(\eps x_\eps(t))\cdot \int q_{\eps}^A(x,t)dx
-(\xi_\eps(t)+A(\eps x_\eps(t))\cdot
\int A(\eps x)|\phi_\eps(x,t)|^2dx\\
& +\frac{1}{2}\int |A(\eps x)|^2|\phi_\eps(x,t)|^2dx +\int A(\eps
x)\cdot q_\eps^A(x,t)dx  \\
& +\frac{1}{2}\sum_{j=1}^m\beta_j \iint \Phi(\eps (x-y))|\phi_\eps^j(x,t)|^2|\phi^j_\eps(y,t)|^2 dx dy \\
& +\frac{1}{2}\sum_{i,j,\,i\neq j}^m\omega_{ij} \iint
\Phi(\eps(x-y))|\phi_\eps^i(x,t)|^2|\phi^j_\eps(y,t)|^2 dx dy+{\mathcal O}(\eps^2).
\end{align*}
\end{lemma}
\begin{proof}
By a change of variable we see that
$\|\psi_{\eps}^j(t)\|_{L^2}^{2}=m_j$ for $j=1,\dots,m$. Hence the
mass of $\psi_{\eps}(t)$ is conserved through the motion. Let
$p_\eps^j(x,t)=\im\big(\bar\phi_\eps^j(x,t)\nabla\phi_\eps^j(x,t)\big)$
for $x\in\R^N$, $t\in\R^+$ and $j=1,\dots,m$ be the $j$-th
magnetic-free momentum. A direct computation yields
\begin{align*}
&\mathcal E(\psi_{\eps}(t))=\frac1{2}\int |\nabla
\phi_{\eps}(x,t)|^{2}dx +
\frac12\sum_{j=1}^m m_j|\xi_\eps(t)+A(\eps x_\eps(t))|^2    \\
& \qquad -\sum_{j=1}^m(\xi_\eps(t)+A(\eps x_\eps(t))\cdot \int p_{\eps}^j(x,t)dx -
\frac{1}{p+1}\sum_{j=1}^m\alpha_j\int |\phi_{\eps}^j(x,t)|^{2p+2}dx\\
& \qquad -\frac{1}{p+1}\sum_{i,j,\,\, i\neq j}^m\gamma_{ij}
\int |\phi_{\eps}^i(x,t)|^{p+1}|\phi_{\eps}^j(x,t)|^{p+1}dx,
\end{align*}
so that we obtain
\begin{align*}
& \mathcal
E(\psi_{\eps}(t))=\frac1{2}\int\left|\frac{1}{\iu}\nabla
\phi_{\eps}(x,t)-A(\eps x)\phi_\eps(x,t)\right|^{2}dx \\
& \qquad -\frac{1}{2}\int |A(\eps x)|^2|\phi_\eps(x,t)|^2 dx
+\sum_{j=1}^m\int A(\eps x)\cdot p_\eps^j(x,t)dx \\
& \qquad
+\frac1 2 \sum_{j=1}^m m_j|\xi_\eps(t)+A(\eps x_\eps(t))|^2
-\sum_{j=1}^m(\xi_\eps(t)+A(\eps x_\eps(t))\cdot \int
p_{\eps}^j(x,t)dx\\
& \qquad -\frac{1}{p+1}\sum_{j=1}^m\alpha_j\int
|\phi_{\eps}^j(x,t)|^{2p+2}dx -\frac{1}{p+1}\sum_{i,j,\,\, i\neq
j}^m\gamma_{ij} \int
|\phi_{\eps}^i(x,t)|^{p+1}|\phi_{\eps}^j(x,t)|^{p+1}dx.
\end{align*}
Then, taking into account the definition of $E_\eps(t)$ and of
$\mathcal H$ and Lemma~\ref{estUNO}, we obtain
\begin{align*}
\mathcal E(\psi_{\eps}(t))-{\mathcal E}(r) &=M{\mathcal H}(t)
-\int V(\eps x)|\phi_{\eps}(x,t)|^{2}dx+\frac12 M|\xi_\eps(t)+A(\eps x_\eps(t))|^2   \\
&-(\xi_\eps(t)+A(\eps x_\eps(t))\cdot \int \sum_{j=1}^m p_{\eps}^j(x,t)dx  \\
& -\frac{1}{2}\int |A(\eps x)|^2|\phi_\eps(x,t)|^2dx+\int A(\eps
x)\cdot \sum_{j=1}^m p_\eps^j(x,t)dx \\
& +\frac{1}{2}\sum_{j=1}^m\beta_j \iint \Phi(\eps (x-y))|\phi_\eps^j(x,t)|^2|\phi^j_\eps(y,t)|^2 dx dy \\
& +\frac{1}{2}\sum_{i,j,\,i\neq j}^m\omega_{ij} \iint
\Phi(\eps(x-y))|\phi_\eps^i(x,t)|^2|\phi^j_\eps(y,t)|^2 dx dy+{\mathcal O}(\eps^2),
\end{align*}
as $\eps$ goes to zero. Finally, since
$p_\eps^j(x,t)=(p_\eps^A)^j(x,t)+A(\eps x)|\phi_\eps^j(x,t)|^2$
and recalling the definition of $q_\eps^A$, we obtain the desired
conclusion.
\end{proof}

\vskip3pt
\noindent
Now let us introduce two functionals in the dual space of $C^2$
\begin{align}
     \int \Pi^1_\eps(x,t) \cdot\varphi(x) dx &=
    \int \varphi(\eps x)\cdot q^A_\eps(x,t)dx-M\varphi(\eps x_\eps(t))\cdot\xi_\eps(t), \quad \forall\varphi\in
C^2(\R^N;\R^N),\label{pigreco1}\\
 \int \Pi^2_\eps(x,t)\varphi(x) dx & =\int \varphi(\eps
x)|\phi_\eps(x,t)|^2dx-M\varphi(\eps x_\eps(t)),  \quad \forall
\varphi\in C^2(\R^N;\R),\label{pigreco2}
\end{align}
for all $t\in\R^+$, where $M$ is given in formula~\eqref{valueofm}.
Moreover, define the function $\Omega_\eps:\R^+\to\R^+$ as $\Omega_\eps(t)=\hat\Omega_\eps(t)+\rho_\eps^A(t)$, where
\begin{align}
\label{defOmega} \hat\Omega_\eps(t)&:=\Big|\int
\Pi^1_\eps(x,t)dx\Big|+\sup_{\|\varphi\|_{C^3}\leq 1} \Big|\int
\Pi^2_\eps(x,t)\varphi(x)dx\Big|+|\gamma_\eps(t)|, \qquad t\in\R^+ \\
\rho_\eps^A(t) &:=\Big|\int \Pi_\eps^1(x,t)\cdot A(x)dx\Big|,\qquad t\in\R^+  \notag
\end{align}
and
$$
\gamma_{\eps}(t):=M\eps x_\eps(t)- \int
\eps x\chi(\eps x)|\phi_{\eps}(x,t)|^{2}dx, \qquad t\in\R^+,
$$
where $\chi\in C^\infty(\R^N)$ is such that $0\leq \chi\leq 1$,
$\chi(x)=1$ in $B(0, \tilde \rho)$ and
$\chi(x)=0$ in $\R^N \setminus B(0, 2\tilde\rho)$,
for a suitable $\tilde\rho>0$ that will be suitably
chosen later.
\vskip3pt

Now we are able to prove an estimate on the energy of $\psi_\eps$.

\begin{lemma}\label{lemmaOmegaepsilon}
Assume that $\Phi=0$ if $m>1$ and
let $\psi_\eps$ be the function defined in formula~\eqref{veps}.
Then there exists a positive constant $C$ independent of $\eps$ such that
\begin{equation*}
0\leq \mathcal E(\psi_{\eps}(t)) -\mathcal E(r) \leq
C\Omega_\eps(t)+{\mathcal O}(\eps^2),
\end{equation*}
for all $t\in\R^+$ and any $\eps>0$.
\end{lemma}
\begin{proof}
We claim that $\Omega_\eps(0)=\mathcal O (\eps^2)$ as $\eps$ goes
to zero. In fact, by definition of $\Omega_\eps$, we have
\begin{equation}\label{omega0}
\Omega_\eps(0)=\Big|\int
\Pi^1_\eps(x,0)dx\Big|+\sup_{\|\varphi\|_{C^3}\leq 1} \Big|\int
\Pi^2_\eps(x,0)\varphi(x)dx\Big|+|\gamma_\eps(0)|+\rho_\eps^A(0).
\end{equation}
First of all, let us estimate the first term in the right-hand side
of~\eqref{omega0}. Taking $\varphi \equiv 1$ in \eqref{pigreco1} and
using \eqref{initialD}, we get
$$
\begin{aligned}
\int \Pi^1_\eps(x,0)dx& =\int q^A_\eps(x,0)dx-M\xi(0)\\
& =\sum_{j=1}^{m}\int \im\big(\bar\phi_\eps^j(x,0)(
\nabla\phi_\eps^j(x,0)-\iu
A(\eps x)\phi_\eps^j(x,0))\big)dx-M\xi_0\\
& =\sum_{j=1}^{m}\int
r_j^2(x-x_\eps(0))\Big[A(x_0)+\xi_0-A(\eps x)\big]dx-M\xi_0\\
& = MA(x_0)-\sum_{j=1}^{m}\int
r_j^2(x-x_\eps(0))A(\eps x)dx\\
& = MA(x_0)-\sum_{j=1}^{m}\int r_j^2(x)A(\eps x+x_0)dx= \mathcal O
(\eps^2),
\end{aligned}
$$
as $\eps$ goes to zero, in light of Lemma~\ref{pote}. In a similar fashion, one gets
$\rho^A_\eps(0)= {\mathcal O}(\eps^2)$. Now consider the
second term in the right-hand side of~\eqref{omega0}. Let
$\varphi\in C^3(\R^N)$ with $\|\varphi\|_{C^3}\leq 1$. Then,
$$\begin{aligned}
\int \Pi^2_\eps(x,0)\varphi(x) \,dx& =\int
\varphi(\eps x)|\phi_\eps(x,0)|^2dx-M\varphi(x(0))\\
& = \sum_{j=1}^{m}\int \varphi(\eps x+x_0)
r_j^2(x)dx-M\varphi(x_0)=\mathcal O (\eps^2)
\end{aligned}$$
as $\eps$ goes to zero, again using Lemma~\ref{pote}. We finally
estimate $\gamma_\eps(0)$. As above we have
$$
\begin{aligned}
\gamma_{\eps}(0)& =Mx(0)-\eps \int
x\chi(\eps x)|\phi_{\eps}(x,0)|^{2}dx\\
& = Mx_0-\eps \sum_{j=1}^{m}\int
x\chi(\eps x)r_j^2(x-x_\eps(0))dx
 = Mx_0-\sum_{j=1}^{m}\int
( \eps x+x_0)\chi(\eps x+x_0)r_j^2(x)dx\\
& = Mx_0-\sum_{j=1}^{m}\int
x_0\chi(x_0)r_j^2(x)dx+\mathcal O (\eps^2)
 = Mx_0\left(1-\chi(x_0)\right)+\mathcal O (\eps^2),
\end{aligned}
$$
thanks to Lemma~\ref{pote}.
Now, from~\cite[Lemma 3.1-3.2]{Keerani2} (where one has to use
the $\delta_a$ at some point $a$ is defined as $\langle \delta_a,\varphi\rangle=\varphi(\eps a)$
for all $\varphi\in C^2(\R^N)$), we learn that
there exist three positive constants $K_{0},\,K_{1},\,K_{2}$ such
that, for all $y,z\in\R^N$,
$K_{1}|\eps y-\eps z|\leq \|\delta_{y}-\delta_{z}\|_{C^{2*}}\leq K_{2}|\eps y-\eps z|$,
provided that $\|\delta_{y}-\delta_{z}\|_{C^{2*}}\leq K_{0}$.
Let then $\tilde \rho=K_1\sup_{\eps\in[0,1]}\sup_{t\in [0,
T_0/\eps]}|\eps x_\eps(t)|+K_0$, where $T_0>0$ is fixed (to be chosen later
on, see Lemma~\ref{lemmaT0}).
Then, in view of the definition of $\chi$, we obtain that
$\gamma_{\eps}(0)=\mathcal O(\eps^2)$ as $\eps$ goes to zero, since
$|x_0|<\tilde\rho$. Hence the claim is proved.
\vskip4pt
\noindent
Now we are ready to prove the assertion of
Lemma~\ref{lemmaOmegaepsilon}. By using Lemma~\ref{estDUE}, the
definition of $\mathcal H$,~\eqref{pigreco1} and~\eqref{pigreco2} we obtain
\begin{align*}
& \mathcal E(\psi_{\eps}(t)) -\mathcal E(r) \\
& = \frac 1 2M |\xi_\eps(t)|^2+MV(\eps x_\eps(t))+M{\mathcal M} \\
&-\int V(\eps
x)|\phi_{\eps}(x,t)|^{2}dx
+\frac12 M|\xi_\eps(t)+A(\eps x_\eps(t))|^2 \\
& -\int \Pi_\eps^1(x,
t)\Big[(\xi_\eps(t)+A(\eps x_\eps(t)))\Big]dx-M\Big[\xi_\eps(t)+A(\eps x_\eps(t))\Big]\cdot\xi_\eps(t)\\
& -(\xi_\eps(t) +A(\eps x_\eps(t))\cdot \Big(\int \Pi_\eps^2(x,
t) A(x)dx+MA(\eps x_\eps(t))\Big)\\
& +\frac{1}{2}\int |A(\eps x)|^2|\phi_\eps(x,t)|^2dx
+\int \Pi_\eps^1(x,t)A(x)dx+MA(\eps x_\eps(t))\cdot \xi_\eps(t)  \\
& +\frac{1}{2}\sum_{j=1}^m\beta_j \iint \Phi(\eps (x-y))|\phi_\eps^j(x,t)|^2|\phi^j_\eps(y,t)|^2 dx dy \\
& +\frac{1}{2}\sum_{i,j,\,i\neq j}^m\omega_{ij} \iint
\Phi(\eps(x-y))|\phi_\eps^i(x,t)|^2|\phi^j_\eps(y,t)|^2 dx dy+{\mathcal O}(\eps^2).
\end{align*}
    Let us set (with the convention that $\omega_{ii}=\beta_i$)
    \begin{equation*}
 \eta_\eps(t)=\left|\sum_{i,j=1}^m \frac{\omega_{ij}}{2}\iint \Phi(\eps (x-y))
|\phi_\eps^i(x,t)|^2|\phi_\eps^j(y,t)|^2dxdy-\Phi(0)\sum_{i,j=1}^m \frac{\omega_{ij}}{2} m_im_j\right|.
\end{equation*}
In turn, using the definition of ${\mathcal M}$, we have
\begin{align*}
& \mathcal E(\psi_{\eps}(t)) -\mathcal E(r) \\
&\leq \eta_\eps(t)+ MV(\eps x_\eps(t)) -\int \Pi_\eps^2(x, t)V(x)dx-MV(\eps x_\eps(t))
+\frac12 M|A(\eps x_\eps(t))|^2 \\
& -\int \Pi_\eps^1(x, t)\Big[\xi_\eps(t)+A(\eps x_\eps(t))\Big]dx
-(\xi_\eps(t) +A(\eps x_\eps( t))\int \Pi_\eps^2(x,
t) A(x)dx-M|A(\eps x_\eps(t))|^2\\
& +\frac{1}{2}\int \Pi_\eps^2(x, t)|A(x)|^2dx +\frac 1 2
M|A(\eps x_\eps(t))|^2+\int \Pi_\eps^1(x, t)A(x)dx+{\mathcal O}(\eps^2)\\
& =\eta_\eps(t)-\int \Pi_\eps^2(x, t)V(x)dx
 -\int \Pi_\eps^1(x, t)\Big[\xi_\eps( t)+A(\eps x_\eps( t))\Big]dx \\
& -(\xi_\eps( t) +A(\eps x_\eps( t))\int \Pi_\eps^2(x,t) A(x)dx  \\
& +\frac{1}{2}\int \Pi_\eps^2(x, t)|A(x)|^2dx +\int \Pi_\eps^1(x, t)A(x)dx+{\mathcal O}(\eps^2)
\leq \eta_\eps(t)+C\Omega_\eps(t)+\mathcal O(\eps^2),
\end{align*}
for $\eps$ sufficiently small. If $m>1$ we assume that $\Phi=0$, and the assertion follows.
If instead $m=1$, observe first that from definition~\eqref{pigreco2},
by choosing $\varphi(x)=\Phi(x-\eps y)$ and $\varphi(y)=\Phi(\eps x_\eps(t)-y)$
respectively, we have
\begin{align*}
 \int \Phi(\eps x-\eps y)|\phi_\eps^1(x,t)|^2dx  &=
\int \Pi^2_\eps(x,t)\Phi(x-\eps y) dx +m_1\Phi(\eps x_\eps(t)-\eps y),  \\
 m_1\int \Phi(\eps x_\eps(t)-\eps y))|\phi_\eps^1(y,t)|^2dy &=
m_1\int \Pi^2_\eps(y,t)\Phi(\eps x_\eps(t)- y) dy +\Phi(0)m_1^2.
\end{align*}
In turn, we have
\begin{align*}
\eta_\eps(t) & \leq C\Big|\int\left[
\int \Pi^2_\eps(x,t)\Phi(x-\eps y) dx +m_1\Phi(\eps x_\eps(t)-\eps y)
\right]|\phi_\eps^1(y,t)|^2dy-\Phi(0)m_1^2\Big| \\
& = C\Big|\int\left[
\int \Pi^2_\eps(x,t)\Phi(x-\eps y) dx\right]|\phi_\eps^1(y,t)|^2dy +m_1\int \Phi(\eps x_\eps(t)-\eps y)
|\phi_\eps^1(y,t)|^2dy-\Phi(0)m_1^2\Big| \\
&\leq Cm_1\sup_{\|\varphi\|_{C^3}\leq 1} \left|\int
\Pi^2_\eps(x,t)\varphi(x)dx\right|+C m_1\left|\int \Pi^2_\eps(y,t)\Phi(\eps x_\eps(t)- y) dy\right|  \\
&\leq C\sup_{\|\varphi\|_{C^3}\leq 1} \left|\int
\Pi^2_\eps(x,t)\varphi(x)dx\right|\leq C\hat \Omega_\eps(t)\leq C\Omega_\eps(t).
\end{align*}
In turn, we conclude that
$$
\mathcal E(\psi_{\eps}(t)) -\mathcal E(r)\leq C\Omega_\eps(t)+\mathcal O(\eps^2)
$$
as $\eps$ goes to zero, for some positive constant $C$. Hence the
proof of Lemma~\ref{lemmaOmegaepsilon} is complete.
\end{proof}

Since the function $\{t\mapsto\Omega_\eps(t)\}$ given
in~\eqref{defOmega} is continuous and recalling that
$\Omega_\eps(0)={\mathcal O}(\eps^2)$ as $\eps\to 0$ (see the proof of
Lemma~\ref{lemmaOmegaepsilon}), for any fixed $T_0>0$ and
$\sigma_0>0$, we can define the time
\begin{equation}
\label{Tepsdef} T^*_\eps:=\sup\big\{t\in
[0,T_0/\eps]:\,\Omega_\eps(s),\,\Gamma_{\psi_\eps(s)}\leq\sigma_0,\,\,\text{for all $s\in
(0,t)$}\big\}>0,
\end{equation}
for any $\eps>0$, where $\Gamma_{\psi_\eps}$ is defined according to~\eqref{gammaUdef} and $\Gamma_{\psi_\eps(0)}=0$.
Now we are able to provide the main result of this section, related to
a representation formula for the solution $\phi_\eps$ of
problem~\eqref{problema}. For the proof, it is enough to adapt the
proof of \cite[Theorem~4.2]{squa}. The fact

\begin{theorem}\label{chiave}
Let $\phi_{\eps}$ be the family of solutions to
problem~\eqref{problema} corresponding to the initial datum~\eqref{initialD}
modelled on a ground state solution $r$ of problem~\eqref{seMF} and let
$(x_\eps(t), \xi_\eps(t))$ be the global solution of~\eqref{DriveS}.
Then there exist positive constants $\eps_{0}$ and $C$, locally bounded functions
$\theta^1_{\eps}, \dots, \theta^m_{\eps}:\R^+\to[0,2\pi)$ and $y_{\eps}:\R^+\to\R^N$ such that
\begin{equation*}
    \phi_\eps^j(x,t)=e^{\iu(\xi_\eps(t)\cdot x+\theta^j_\eps(t)+A(\eps x_\eps(t))
    \cdot (x-x_\eps(t))}r_j(x-y_\eps(t))+\omega^j_\eps(t),
\end{equation*}
where $\|\omega^j_\eps(t)\|_{H^1}\leq C\sqrt{\Omega_\eps(t)}+{\mathcal O}(\eps)$,
for all $\eps\in(0,\eps_{0})$, $t\in [0,T_{\eps}^{*})$ and $j=1,
\dots, m$.
\end{theorem}

\section{Density and momentum identities}
\label{masssection}

This section is devoted to some important identities involving the
momentum $p_{\eps}^A$ and the total magnetic momentum $q_{\eps}^A$
related to problem~\eqref{problema}.

\begin{proposition}\label{identita}
Let $\phi_\eps$ be the solution to problem~\eqref{problema}
corresponding to the initial datum~\eqref{initialD}. Then the
following identities hold true
\begin{equation}
\label{identitafiepsilon} \frac{\partial
|\phi_{\eps}^j|^{2}}{\partial t}(x,t)=-{\rm div}_{x}\,
(p_{\eps}^A)^j(x,t),\quad x\in\R^N,\, t\in\R^+,\, j=1, \dots, m,
\end{equation}
\begin{equation}\label{identitamomento}
\begin{aligned}
\int \frac{\partial q^A_{\eps}}{\partial t}(x,t)dx & = -\int
q_{\eps}^A(x,t)\times \eps B(\eps x)dx -\int \eps\nabla
V(\eps x)|\phi_{\eps}(x,t)|^{2}dx\\
&\quad  + \sum_{j=1}^m\beta_j\iint
\eps\nabla\Phi(\eps(x-y))|\phi_\eps^j(y)|^2|\phi_\eps^j(x)|^2dxdy\\
&\quad  + \sum_{i,j=1, i\neq j}^m\omega_{ij}\iint
\eps\nabla\Phi(\eps(x-y))|\phi_\eps^i(y)|^2|\phi_\eps^j(x)|^2dxdy,
\end{aligned}
\end{equation}
for $t\in\R^+$, where $B=\nabla\times A$ is the magnetic field
associated with $A$.
\end{proposition}

\begin{proof}
    The proof follows the lines of the corresponding proof in~\cite{squa}
    for the scalar case without the presence of nonlocal potentials.
By formula~\eqref{paepsilon}, for any $j=1,\dots,m$,  $(p_{\eps}^A)^j$ is the vector whose
components, which we denote by $(p_{\eps}^A)^j_\ell$, are given by
$(p_{\eps}^A)^j_\ell=\im\big(\bar\phi_\eps^j(x,t)(
\partial_\ell\phi_\eps^j(x,t)-\iu A_\ell(\eps x)\phi_\eps^j(x,t))\big)$, for $\ell=1, \dots N$.
Let us fix $j=1, \dots, m$. Hence
\begin{align*}
-{\rm div}_{x}\,(p_{\eps}^A)^j(x,t)& =-\sum_{\ell=1}^N\im\big(\partial_\ell\bar\phi_\eps^j(x,t)
(\partial_\ell\phi_\eps^j(x,t)-\iu A_\ell(\eps x)\phi_\eps^j(x,t)\big)   \\
& \qquad -\sum_{\ell=1}^N\im\big(\bar\phi_\eps^j(x,t) (
\partial_{\ell\ell}^2\phi_\eps^j(x,t)-\iu \partial_\ell
A_\ell(\eps x)\phi_\eps^j(x,t)
-\iu A_\ell(\eps x)\partial_\ell\phi_\eps^j(x,t)\big) \\
& =2A(\eps
x)\cdot\re\big(\nabla\bar\phi_\eps^j(x,t)\phi_\eps^j(x,t)\big)
-\im\big(\bar\phi_\eps^j(x,t)\Delta\phi_\eps^j(x,t))+{\rm
div}_x\,A(\eps x)|\phi_\eps^j(x,t)|^2.
\end{align*}
Moreover, using \eqref{problema} and taking into account the
definition of $L_A$, we get
$$\begin{aligned}
\frac{\partial |\phi_{\eps}^j|^{2}}{\partial t}(x,t)& =
2\im\left(\bar\phi_{\eps}^j(x,t)
\big(L_A\phi_\eps^j(x,t)
+V(\eps x)\phi_\eps^j(x,t)-|\phi_\eps(x,t)|^{2p}_j\phi_\eps^j(x,t)\right.\\
& \qquad
\left.-\Phi(\eps x)*|\phi_\eps|^2_j\phi_\eps^j(x,t)\big)\right)\\
& =-\im\big(\bar\phi_\eps^j(x,t)\Delta\phi_\eps^j(x,t)) +2 A(\eps
x)\cdot\re\big(\phi_\eps^j(x,t)\nabla\bar\phi_\eps^j(x,t)\big)
+{\rm div}_x\,A(\eps x)|\phi_\eps^j(x,t)|^2,
\end{aligned}$$
so that identity~\eqref{identitafiepsilon} holds true.
Now let us prove the second one. By definition of the total
magnetic momentum $q_\eps^A$, for any $\ell=1,\dots, N$, we have
$$\begin{aligned}
\frac{\partial (q_\eps^A)_\ell}{\partial t} =\sum_{j=1}^m
\frac{\partial (p_{\eps}^A)_\ell^j}{\partial t}
&=\sum_{j=1}^m\left(\im(\partial_t\bar{\phi}_\eps^j\partial_\ell\phi_\eps^j)
+\im(\partial_\ell(\phi_\eps^j\partial_t\phi_\eps^j))\right)
-\sum_{j=1}^m\im(\partial_\ell\bar{\phi}_\eps^j\partial_t\phi_\eps^j)
-A_\ell(\eps x)\sum_{j=1}^m\frac{\partial
|\phi_\eps^j|^2}{\partial
t}\\
& =
2\sum_{j=1}^m\im(\partial_t\bar{\phi}_\eps^j\partial_\ell\phi_\eps^j)
-\sum_{j=1}^m\im\left(\partial_\ell(\bar{\phi}_\eps^j\partial_t\phi_\eps^j)\right)
-A_\ell(\eps x)\sum_{j=1}^m\frac{\partial
|\phi_\eps^j|^2}{\partial t},
\end{aligned}$$
and so, integrating over $\R^N$, it is easy to see that
\begin{equation}
    \label{qepsA}
\int \frac{\partial (q_\eps^A)_\ell}{\partial t} dx
=2\sum_{j=1}^m\int
\im(\partial_t\bar{\phi}_\eps^j\partial_\ell\phi_\eps^j)dx-\sum_{j=1}^m\int
\im\left(\partial_\ell(\bar{\phi}_\eps^j\partial_t\phi_\eps^j)\right)dx
  -\sum_{j=1}^m\int A_\ell(\eps x)\frac{\partial
|\phi_\eps^j|^2}{\partial t}dx.
\end{equation}
Let us consider the first term in the right-hand side of
\eqref{qepsA}. Conjugating the equation, multiplying it by
$2i\partial_\ell\phi_\eps^j$, $\ell=1, \dots, N$, and taking the
imaginary part, we have
$$\begin{aligned} 2
\im(\partial_{t}\bar{\phi_{\eps}}^j\partial_\ell\phi_{\eps}^j)
&=-\re(\Delta \bar{\phi_{\eps}}^j\partial_\ell\phi_{\eps}^j)
+2A(\eps x)\cdot\im(\nabla
\bar\phi_\eps^j\partial_\ell\phi_\eps^j)
+|A(\eps x)|^2\re(\bar{\phi_{\eps}}^j\partial_\ell\phi_{\eps}^j)  \\
& \qquad +{\rm div}_xA(\eps
x)\im(\bar\phi_\eps^j\partial_\ell\phi_\eps^j)+
2V(\eps x)\re(\bar{\phi_{\eps}}^j\partial_\ell\phi_{\eps}^j)\\
& \qquad
-2\re(|\phi_{\eps}|^{2p}_j\bar{\phi_{\eps}}^j\partial_\ell\phi_{\eps}^j)
- 2\re\left((\Phi(\eps x)*|\phi_\eps|_j^2)\bar \phi_\eps^j\partial_\ell\phi_\eps^j\right)\\
&=-\sum_{i=1}^{m}\re\left(\partial_i(\partial_i
\bar{\phi_{\eps}}^j\partial_\ell\phi_{\eps}^j)\right)
+\sum_{i=1}^{m}\partial_\ell\left(\frac{|\partial_i\phi_\eps^j|^2}{2}\right)\\
& \qquad +2A(\eps x)\cdot\im(\nabla
\bar\phi_\eps^j\partial_\ell\phi_\eps^j)
+|A(\eps x)|^2\re(\bar{\phi_{\eps}}^j\partial_\ell\phi_{\eps}^j)  \\
& \qquad +{\rm div}_xA(\eps
x)\im(\bar\phi_\eps^j\partial_\ell\phi_\eps^j)+
\partial_\ell\left(V(\eps x)|\phi_{\eps}^j|^2\right)
-\eps \partial_\ell V(\eps x)|\phi_{\eps}^j|^2\\
& \qquad
-\frac{\alpha_j}{p+1}\partial_\ell\big(|\phi_{\eps}^j|^{2p+2}\big)
-2\sum_{i=1, i\not =
j}^m\gamma_{ij}|\phi_\eps^i|^{p+1}|\phi_\eps^j|^{p-1}
\re(\bar{\phi_{\eps}}^j\partial_\ell\phi_{\eps}^j)\\
& \qquad - 2\beta_j \re\left((\Phi(\eps x)*|\phi_\eps^j|^2)\bar
\phi_\eps^j\partial_\ell\phi_\eps^j\right)- 2\sum_{i=1, i\not =
j}^m\omega_{ij} \re\left((\Phi(\eps x)*|\phi_\eps^i|^2)\bar
\phi_\eps^j\partial_\ell\phi_\eps^j\right).
\end{aligned}$$
Hence, integrating over $\R^N$ and using the
$H^2$-regularity of the functions involved, for all $\ell=1, \dots,
N$ we obtain the following identity
$$\begin{aligned}
2 \int
\im(\partial_{t}\bar{\phi_{\eps}}^j\partial_\ell\phi_{\eps}^j)dx
&=2\int A(\eps x)\cdot\im(\nabla
\bar\phi_\eps^j\partial_\ell\phi_\eps^j)dx
+\int|A(\eps x)|^2\re(\bar{\phi_{\eps}}^j\partial_\ell\phi_{\eps}^j)dx  \\
& \qquad +\int{\rm div}_xA(\eps
x)\im(\bar\phi_\eps^j\partial_\ell\phi_\eps^j)dx
-\eps \int\partial_\ell V(\eps x)|\phi_{\eps}^j|^2dx\\
& \qquad -2\sum_{i=1, i\not = j}^m\gamma_{ij}
\int |\phi_\eps^i|^{p+1}|\phi_\eps^j|^{p-1}\re(\bar{\phi_{\eps}}^j\partial_\ell\phi_{\eps}^j)dx\\
& \qquad - 2\beta_j \int\re\left((\Phi(\eps x)*|\phi_\eps^j|^2)\bar
\phi_\eps^j\partial_\ell\phi_\eps^j\right)dx\\
& \qquad - 2\sum_{i=1, i\not = j}^m\omega_{ij}
\int\re\left((\Phi(\eps x)*|\phi_\eps^i|^2)\bar
\phi_\eps^j\partial_\ell\phi_\eps^j\right)dx.
\end{aligned}$$
Notice that
$$
\int {\rm div}_xA(\eps
x)\im(\bar\phi_\eps^j\partial_\ell\phi_\eps^j)dx+ 2\!\!\int A(\eps
x)\cdot\im(\nabla \bar\phi_\eps^j\partial_\ell\phi_\eps^j)dx=\eps
\sum_{i=1}^m \int
\partial_\ell A_i(\eps x)\im\big(\bar\phi_\eps^j\partial_i\phi_\eps^j\big)dx.
$$
Moreover, thanks to the regularity of $\phi_\eps^j$, we have
$$
\begin{aligned}
\sum_{i,j=1,\, i\neq j}^{m}\int |\phi_\eps^i|^{p+1}|\phi_\eps^j|^{p-1}\re(\bar{\phi_{\eps}}^j\partial_\ell\phi_{\eps}^j)dx
& = \sum_{i,j=1,\, i\neq j}^{m}\int |\phi_\eps^i|^{p+1}|\phi_\eps^j|^{p-1}\partial_\ell\left(\frac{|\phi_\eps^j|^2}{2}\right)dx\\
& = \frac{1}{p+1}\sum_{i,j=1,\, i\neq j}^{m}\int |\phi_\eps^i|^{p+1}\partial_\ell\left(|\phi_\eps^j|^{p+1}\right)dx\\
& = \frac{1}{p+1}\sum_{i,j=1,\, i<j}^{m}\int
\partial_\ell\left(|\phi_\eps^i|^{p+1}|\phi_\eps^j|^{p+1}\right)dx=0,
\end{aligned}$$
and
$$\begin{aligned}
\int\re\left((\Phi(\eps x)*|\phi_\eps^j|^2)\bar
\phi_\eps^j\partial_\ell\phi_\eps^j\right)dx
& =\iint \Phi(\eps(x-y))|\phi_\eps^j(y)|^2\re\big(\bar\phi_\eps^j(x)\partial_\ell\phi_\eps^j(x)\big)dy dx\\
& = \iint \Phi(\eps(x-y))|\phi_\eps^j(y)|^2\partial_\ell\left(\frac{|\phi_\eps^j(x)|^2}{2}\right)dy dx\\
& =- \frac 1 2\iint
\eps\partial_\ell\Phi(\eps(x-y))|\phi_\eps^j(y)|^2|\phi_\eps^j(x)|^2dxdy
\end{aligned}$$
for all $\ell=1, \dots, N$. While, with the same arguments, we get
$$\begin{aligned}
\sum_{i=1, i\not =
j}^m\omega_{ij}\int\re\big((\Phi(\eps x)&*|\phi_\eps^i|^2)\bar
\phi_\eps^j\partial_\ell\phi_\eps^j\big)dx =- \frac 1 2\sum_{i=1,
i\not = j}^m\omega_{ij}\iint
\eps\partial_\ell\Phi(\eps(x-y))|\phi_\eps^i(y)|^2|\phi_\eps^j(x)|^2dxdy.
\end{aligned}$$
Hence, it is easy to see that
\begin{equation}\label{termine1}
\begin{aligned}
2 \int
\im(\partial_{t}\bar{\phi_{\eps}}^j \partial_\ell\phi_{\eps}^j)dx
& =  \sum_{i=1}^m \int
\eps\partial_\ell
A_i(\eps x)\im\big(\bar\phi_\eps^j\partial_i\phi_\eps^j\big)dx
+\int|A(\eps x)|^2\partial_\ell\left(\frac{|\phi_\eps^j|^2}{2}\right)dx\\
&  - \int \eps\partial_\ell V(\eps x)|\phi_\eps^j|^2dx -
\beta_j\iint
\eps\partial_\ell\Phi(\eps(x-y))|\phi_\eps^j(y)|^2|\phi_\eps^j(x)|^2dxdy\\
&  -\sum_{i=1, i\not =j}^m\omega_{ij}\iint
\eps\partial_\ell\Phi(\eps(x-y))|\phi_\eps^i(y)|^2|\phi_\eps^j(x)|^2dxdy \\
& = \sum_{i=1}^m \int
\eps\partial_\ell
A_i(\eps x)\im\big(\bar\phi_\eps^j\partial_i\phi_\eps^j\big)dx
+\sum_{i=1}^m\int \eps A_i(\eps x)\partial_\ell A_i(\eps x)|\phi_\eps^j|^2dx\\
&  -  \int \eps\partial_\ell V(\eps x)|\phi_\eps^j|^2dx-
\beta_j\iint
\eps\partial_\ell\Phi(\eps(x-y))|\phi_\eps^j(y)|^2|\phi_\eps^j(x)|^2dxdy\\
&  -\sum_{i=1, i\not =j}^m\omega_{ij}\iint
\eps\partial_\ell\Phi(\eps(x-y))|\phi_\eps^i(y)|^2|\phi_\eps^j(x)|^2dxdy
\end{aligned}
\end{equation}
for all $\ell=1, \dots, N$. As for the second term in~\eqref{qepsA},
using again the regularity of $\phi_\eps^j$, we get
$\int \im\left(\partial_\ell(\bar{\phi}_\eps^j\partial_t\phi_\eps^j)\right)dx=0$,
for any $\ell=1, \dots, N$. Finally, as for the third
term in the right-hand side of~\eqref{qepsA}, by \eqref{identitafiepsilon} we get
\begin{equation}
    \label{termine3}
\begin{aligned}
\int A_\ell(\eps x)\frac{\partial |\phi_\eps^j|^2}{\partial
t}(x,t)dx &=-\int A_\ell(\eps x){\rm div}_{x}\,
(p_{\eps}^A)^j(x,t)dx \\
& =  \sum_{i=1}^{m}\int \eps\partial_i A_\ell(\eps x) (p_{\eps}^A)^j_i(x,t)dx\\
& =  \sum_{i=1}^{m}\int \eps\partial_i A_\ell(\eps x)
\im(\bar\phi_{\eps}^j(x, t)\left(\partial_i
\phi_{\eps}^j(x,t)-iA_i(\eps x)\phi_{\eps}^j(x,t))\right)\\
& = \int \sum_{i=1}^m \eps\partial_i A_\ell(\eps x)
\im\big(\bar\phi_\eps^j(x,t)\partial_i\phi_\eps^j(x,t)\big)dx \\
&  -\int \sum_{i=1}^m \eps A_i(\eps x)\partial_i A_\ell(\eps
x)|\phi_\eps^j(x,t)|^2dx
\end{aligned}
\end{equation}
for any $\ell=1, \dots, N$.
Then \eqref{qepsA}-\eqref{termine3} yield
$$\begin{aligned}
\int \frac{\partial (q_\eps^A)_\ell}{\partial t}(x,t) dx & =
\sum_{i,j=1}^m \int \eps\left(\partial_\ell A_i(\eps x)-\partial_i
A_\ell(\eps x)\right)\im\big(\bar\phi_\eps^j\partial_i\phi_\eps^j\big)dx\\
& \qquad + \sum_{i,j=1}^m\int \eps A_i(\eps x)\left(\partial_\ell
A_i(\eps x)-\partial_i A_\ell(\eps x)\right)
|\phi_\eps^j|^2dx\\
& \qquad - \sum_{j=1}^m\int \eps\partial_\ell V(\eps
x)|\phi_\eps^j|^2dx- \sum_{j=1}^m\beta_j\iint
\eps\partial_\ell\Phi(\eps(x-y))|\phi_\eps^j(y)|^2|\phi_\eps^j(x)|^2dxdy\\
& \qquad -\sum_{i=1, i\not =j}^m\omega_{ij}\iint
\eps\partial_\ell\Phi(\eps(x-y))|\phi_\eps^i(y)|^2|\phi_\eps^j(x)|^2dxdy\\
& = -\int \left(q_{\eps}^A(x,t)\times \eps B(\eps x)\right)_\ell dx
-  \int \eps\partial_\ell V(\eps x)|\phi_\eps|^2dx\\
& \qquad - \sum_{j=1}^m\beta_j\iint
\eps\partial_\ell\Phi(\eps(x-y))|\phi_\eps^j(y)|^2|\phi_\eps^j(x)|^2dxdy\\
& \qquad -\sum_{i=1, i\not =j}^m\omega_{ij}\iint
\eps\partial_\ell\Phi(\eps(x-y))|\phi_\eps^i(y)|^2|\phi_\eps^j(x)|^2dxdy
\end{aligned}$$
for any $\ell=1, \dots, N$, so that \eqref{identitamomento} is proved.
\end{proof}

\begin{remark}
Taking into account the definition of $q_{\eps}^A$, by
\eqref{identitafiepsilon} easily follows
$$\frac{\partial
|\phi_{\eps}|^{2}}{\partial t}(x,t)=-{\rm div}_{x}\,
q_{\eps}^A(x,t),\quad x\in\R^N,\, t\in\R^+,$$
which is consistent
with the conservation's laws for the nonlinear Schr\"odinger
equation.
\end{remark}

We now give some estimates on the momentum $p_{\eps}^A$
and the total magnetic momentum $q_{\eps}^A$ related to
problem~\eqref{problema}.

\begin{lemma}\label{tec1}
Let $\phi_\eps$ be the solution of problem~\eqref{problema}
corresponding to the initial datum~\eqref{initialD} and let $(x_\eps(t),
\xi_\eps(t))$ be the global solution to~\eqref{DriveS}. Then, in the notational framework
of Theorem~\ref{chiave}, there exist $\eps_0>0$ and $C>0$ such that
\begin{equation*}
\big\| |\phi_\eps^j(x,t)|^2dx-m_j\delta_{y_{\eps}(t)}\big\|_{C^{2*}}
+
\big\|q_{\eps}^{A(\eps x_\eps(t))}(x,t)dx-M\xi_\eps(t)\delta_{y_{\eps}(t)}\big\|_{C^{2*}}\leq
C\Omega_\eps(t)+{\mathcal O}(\eps^2),
\end{equation*}
for every $t\in[0,T_{\eps}^{*})$ and $\eps\in(0,\eps_{0})$ and for
all $j=1, \dots, m$, where $T_\eps^*$ is given in~\eqref{Tepsdef}.
\end{lemma}

\begin{proof}
For any $v\in H^{1}(\R^N)$, we have the formula $|\nabla
|v||^{2}=|\nabla v|^{2}-\frac{|\im(\bar v\nabla
v)|^{2}}{|v|^{2}}$. Then, by virtue of
Lemma~\ref{lemmaOmegaepsilon}, it follows that
\begin{equation*}
0\leq \mathcal E(|\psi_{\eps}|) -\mathcal E(r) +
\frac{1}{2}\sum_{j=1}^m\int\frac{|\im(\bar \psi_{\eps}^j\nabla
\psi_{\eps}^j)|^{2}}{|\psi_{\eps}^j|^{2}}dx\leq
C\Omega_\eps(t)+{\mathcal O}(\eps^2),
\end{equation*}
for all $t\in \R^+$ and $\eps>0$. Moreover, since
$\||\psi_{\eps}^j|\|_{L^2}=\|r_j\|_{L^2}$ for all $j=1, \dots, m$
and $\mathcal E(|\psi_{\eps}|)\geq \mathcal E(r)$
by means of~\eqref{variatcaract-r}, we have
\begin{equation}\label{dis1}
    \int\frac{|\im(\bar \psi_{\eps}^j\nabla \psi_{\eps}^j)|^{2}}{|\psi_{\eps}^j|^{2}}dx\leq C\Omega_\eps(t)+{\mathcal O}(\eps^2),
\end{equation}
for every $t\in \R^+$ and $\eps>0$ and for all $j=1, \dots, m$.
Following the blueprint of~\cite[Lemma~6.1]{squa}, we get the assertion
(see also \cite{squamont}).
\end{proof}

\vskip4pt
\noindent
Lemma~\ref{tec1} allows us to prove the following result
on the distance between the point $y_{\eps}(t)$ found
out in Theorem~\ref{chiave} and the trajectory $x_\eps(t)$. For the proof,
follow the blueprint of~\cite[Lemma~6.3]{squa}.

\begin{lemma}\label{lemmaT0}
    In the notational framework of Theorem~\ref{chiave} there exist $\eps_0>0$ and $T_0>0$ (cf.~the
    definition of $T_{\eps}^{*}=T_{\eps}^{*}(T_0)$) such that
$$
\big\|\delta_{x_\eps(t)}-\delta_{y_\eps(t)}\big\|_{C^{2*}}\leq
C|\eps x_\eps(t)-\eps y_{\eps}(t)|\leq C\Omega_\eps(t)+{\mathcal O}(\eps^2),
$$
for all $\eps\in(0,\eps_{0})$, $t\in [0,T_{\eps}^{*})$,
where $T_{\eps}^{*}$ defined as in \eqref{Tepsdef}.
\end{lemma}

\noindent
Next, we state a strengthened version of
Lemma~\ref{tec1}, obtained thanks to Lemma~\ref{lemmaT0}. Follow
the blueprint of~\cite[Lemma~6.4]{squa} for a proof.

\begin{lemma}\label{lemmaT0+}
    Let $T_0$ be as in Lemma~\ref{lemmaT0}.
Let $\phi_{\eps}$ be the family of solutions to
problem~\eqref{problema} with initial datum~\eqref{initialD} and let
$(x_\eps(t), \xi_\eps(t))$ be the global solution of~\eqref{DriveS}.
Then there exist $\eps_{0}>0$ and $C>0$ such that
\begin{equation*}
\big\| |\phi_{\eps}^j(x,t)|^2dx-m_j\delta_{x_\eps(t)}\big\|_{C^{2*}}
+
\big\|q_{\eps}^{A}(x,t)dx-M\xi_\eps(t)\delta_{x_\eps(t)}\big\|_{C^{2*}}\leq
C\Omega_\eps(t)+{\mathcal O}(\eps^2),
\end{equation*}
for all $\eps\in(0,\eps_{0})$, $t\in [0,T_{\eps}^{*})$.

\noindent
In particular, by the definition
of $\Omega_\eps$, there exists $\delta>0$ with
\begin{equation}
	\label{secDDIn}
\big\| |\phi_{\eps}^j(x,t)|^2dx-m_j\delta_{x_\eps(t)}\big\|_{C^{2*}}
+
\big\|q_{\eps}^{A}(x,t)dx-M\xi_\eps(t)\delta_{x_\eps(t)}\big\|_{C^{2*}}\leq
C\hat\Omega_\eps(t)+{\mathcal O}(\eps^2),
\end{equation}
for all $\eps\in(0,\eps_{0})$ and $t\in [0,T_{\eps}^{*})$, provided that $\|A\|_{C^2}<\delta$.
\end{lemma}

\begin{remark}
In Lemma~\ref{lemmaT0+}, while the $C^{2*}$-norm control holds on  $\Pi_\eps^j=|\phi_{\eps}^j(x,t)|^2dx-m_j\delta_{x_\eps(t)}$
for each $j=1,\dots,m$, the control on the momentum holds for the total
magnetic momentum $q_{\eps}^{A}(x,t)$. This is in fact natural, since looking
at the second identity in Proposition~\ref{identita}, it is clear that
it cannot hold for each individual $(p_\eps^A)^j$, unless some other (disturbing)
integral terms are added to the formula.
\end{remark}

\section{Uniform estimation of $\Omega_\eps$}
\label{error-estimate}

Before proving the main result we give an estimate
showing that the quantity $\Omega_\eps(t)$ can be made small at
the order ${\mathcal O}(\eps^2)$, uniformly on finite
time intervals, as $\eps$ goes to zero.

\begin{lemma}\label{OOmega}
Let $T_0$ be as in Lemma~\ref{lemmaT0} and $\eps_0,\delta$ as in Lemma~\ref{lemmaT0+}. Then
there exists $C>0$ such that $\hat\Omega_\eps(t)\leq C\eps^2$, for all $\eps\in (0,\eps_0)$
and $t\in [0,T_{\eps}^{*})$.

\noindent
In addition, if we assume that $\|A\|_{C^2}<\delta$ for $\delta>0$ sufficiently enough,
then $\Omega_\eps(t)\leq C\eps^2$, for all $\eps\in (0,\eps_0)$
and $t\in [0,T_{\eps}^{*})$.
\end{lemma}
\begin{proof}
By the definition of $\Pi^1_\eps$, Lemma~\ref{lemmaT0+},
Proposition~\ref{identita} and system~\eqref{DriveS}, we obtain
\begin{align*}
\Big|\int \frac{d}{dt}\Pi^1_\eps(x,t)dx\Big| &
=\Big|\int \frac{\partial q^A_\eps}{\partial t}(x,t)dx-M\dot \xi_\eps(t)\Big| \\
& =\Big|\int q_{\eps}^A(x,t)\times \eps B(\eps x)dx+\int \eps\nabla V(\eps x)|\phi_{\eps}(x,t)|^{2}dx \\
& +\sum_{j=1}^m\beta_j\iint \eps\nabla\Phi(\eps(x-y))|\phi_\eps^j(y)|^2|\phi_\eps^j(x)|^2dxdy\\
& +\sum_{i,j=1, i\neq j}^m\omega_{ij}\iint \eps\nabla\Phi(\eps(x-y))|\phi_\eps^i(y)|^2|\phi_\eps^j(x)|^2dxdy\\
& -M\eps\nabla V(\eps x_\eps(t))-M\eps\xi_\eps(t)\times B(\eps x_\eps(t))\Big|.
\end{align*}
If $m>1$, we do not have to manage the nonlocal terms, since $\Phi\equiv 0$. If instead $m=1$,
recalling that $\nabla \Phi(0)=0$, by Lemma~\ref{lemmaT0+}
and arguing as at the end of the proof of Lemma~\ref{lemmaOmegaepsilon}, we get
\begin{equation}
	\label{estsmm0}
\Big|\iint \eps\nabla \Phi(\eps(x-y))|\phi_\eps^1(y)|^2|\phi_\eps^1(x)|^2dx dy\Big|
\leq \eps\big[ C\hat\Omega_\eps(t)+{\mathcal O}(\eps^2) \big],
\end{equation}
for some positive constant $C$,
for all $\eps\in(0,\eps_{0})$ and $t\in [0,T_{\eps}^{*})$. In turn, it holds
\begin{align}
	\label{estsmm1}
&   \Big|\int \frac{d}{dt}\Pi^1_\eps(x,t)dx\Big| \\
&\leq \Big|\int q_{\eps}^A(x,t)\times \eps B(\eps x)dx+\int \eps\nabla V(\eps x) |\phi_{\eps}(x,t)|^{2} dx  \notag \\
& -\int M\eps\nabla V(\eps x)\delta_{x_{\eps}(t)}dx-\int M\eps\xi_\eps(t)\times B(\eps x)\delta_{x_\eps(t)}dx\Big|
+ \eps\big[ C\hat\Omega_\eps(t)+{\mathcal O}(\eps^2) \big] \notag \\
&\leq  \eps\Big|\int \big(q_{\eps}^A(x,t)-M\xi_\eps(t)\delta_{x_\eps(t)}\big)\times  B(\eps x)dx\Big|  \notag \\
& +\eps\Big|\int \nabla V(\eps x)\big(|\phi_{\eps}(x,t)|^{2}-M\delta_{x_\eps (t)}\big)dx\Big|
+ \eps\big[ C\hat\Omega_\eps(t)+{\mathcal O}(\eps^2) \big] \notag \\
&\leq C\eps\big\|q_{\eps}^A(x,t)dx-M\xi_\eps(t)\delta_{x_\eps(t)}\big\|_{C^{2*}}
+C\eps\big\| |\phi_{\eps}(x,t)|^{2}dx-M\delta_{x_\eps(t)} \big\|_{C^{2*}}   \notag\\
\noalign{\vskip3pt}
& + \eps\big[ C\hat\Omega_\eps(t)+{\mathcal O}(\eps^2) \big]
\leq \eps\big[ C\hat\Omega_\eps(t)+{\mathcal O}(\eps^2) \big],  \notag
\end{align}
for all $\eps\in (0,\eps_0)$ and $t\in [0,T^*_\eps)$. Hence, recalling that
$\Omega_\eps(0)=\mathcal O(\eps^2)$ as $\eps$ goes to zero,
\begin{equation}\label{p1}
\left|\int \Pi^1_\eps(x,t)dx\right|
\leq\left|\int \Pi^1_\eps(x,0)dx\right|+\int_{0}^t
\left|\int\frac{d}{dt}\Pi^1_\eps(x,\tau)dx\right|d\tau
\leq C\eps^2(1+\eps t)+C\eps\int_{0}^t \hat\Omega_\eps(\tau)d\tau,
\end{equation}
for all $\eps\in (0,\eps_0)$ and $t\in [0,T^*_\eps)$.
Now, let $\varphi\in C^3(\R^N)$ such that $\|\varphi\|_{C^3}\leq 1$.
Again in light of Proposition~\ref{identita} and Lemma~\ref{lemmaT0+}, we have
\begin{align}
	\label{estsmm2}
\Big|\int \frac{d}{dt}\Pi^2_\eps(x,t)\varphi(x)dx\Big| &
=\Big|\int \varphi(\eps x)\frac{\partial}{\partial t}
|\phi_\eps(x,t)|^2dx-M\eps\nabla\varphi(\eps x_\eps(t))\cdot\xi_\eps(t)\Big| \\
& =\Big|- \int \varphi(\eps x)\,{\rm div}_{x}\, q_{\eps}^A(x,t)dx-M\eps\nabla\varphi(\eps x_\eps(t))\cdot\xi_\eps(t)\Big| \notag\\
& =\Big|\int \eps\nabla \varphi(\eps x)\cdot q_{\eps}^A(x,t)dx-
\int M\eps \nabla\varphi(\eps x)\cdot\xi_\eps(t)\delta_{x_\eps(t)}dx\Big|   \notag \\
& =\Big|\int \eps\nabla \varphi(\eps x)\cdot\big (q_{\eps}^A(x,t)-M\xi_\eps(t)\delta_{x_\eps(t)}\big)dx\Big|  \notag\\
&\leq C\eps\big\|q_{\eps}^A(x,t)dx-M\xi_\eps(t)\delta_{x_\eps(t)}\big\|_{C^{2*}}
\leq \eps\big[ C\hat\Omega_\eps(t)+\mathcal O(\eps^2)\big],
\notag
\end{align}
for all $\eps\in (0,\eps_0)$ and $t\in [0,T^*_\eps)$. Thus, arguing
as above, we get
\begin{equation}\label{p2}
\sup_{\|\varphi\|_{C^3}\leq 1}\Big|\int
\Pi^2_\eps(x,t)\varphi(x)dx\Big| \leq C\eps^2(1+\eps t)+C\eps\int_{0}^t \hat\Omega_\eps(\tau)d\tau,
\end{equation}
for all $\eps\in (0,\eps_0)$ and $t\in [0,T^*_\eps)$.
Finally, again via Proposition~\ref{identita} and Lemma~\ref{lemmaT0+}, we have
\begin{align}
	\label{estsmm3}
\big|\dot\gamma_{\eps}(t)\big| &
=\Big|M\eps\xi_\eps(t)+\int \eps x\chi(\eps x)\,{\rm div}_x q^A_\eps(x,t)dx\Big|  \\
&=\Big|M\eps\xi_\eps(t)-\int \nabla(\eps x\chi(\eps x))\cdot q^A_\eps(x,t)dx\Big|  \notag \\
&=\eps\Big|\int \nabla(x\chi(\eps x))M\xi_\eps(t)\delta_{x_\eps(t)}dx
-\int \nabla(x\chi(\eps x))\cdot q^A_\eps(x,t)dx\Big|  \notag \\
\noalign{\vskip2pt} &\leq \eps C\big\|q^A_\eps(x,t)dx-M\xi_\eps(t)\delta_{x_\eps (t)}\big\|_{C^{2*}}
\leq \eps\big[C\hat\Omega_\eps(t)+\mathcal O(\eps^2)\big], \notag
\end{align}
which implies
\begin{equation}
\label{p3} |\gamma_\eps(t)|\leq
C\eps^2(1+\eps t)+C\eps\int_0^t \hat\Omega_\eps(\tau)d\tau,
\end{equation}
for all $\eps\in (0,\eps_0)$ and $t\in [0,T^*_\eps)$. Collecting the above
inequalities, recalling the definition of
$\hat\Omega_\eps(t)$ and taking into account that, for $t<T^*_\eps$,
by the definition of $T_\eps^*$ it holds $\eps t< \eps T_\eps^*\leq T_0$, we get
\begin{equation*}
\hat\Omega_\eps(t) \leq C\eps^2(1+\eps t)+C\eps\int_{0}^t \hat\Omega_\eps(\tau) d\tau \\
 \leq C\eps^2+C\eps\int_{0}^t \hat\Omega_\eps(\tau) d\tau
\end{equation*}
for all $\eps\in (0,\eps_0)$ and $t\in [0,T^*_\eps)$. Hence, Gronwall Lemma
yields
$$
\hat\Omega_\eps(t)\leq C\eps^2 e^{\eps t}\leq C\eps^2,
$$
for all $\eps\in (0,\eps_0)$ and $t\in [0,T^*_\eps)$, which gives the assertion.
Finally, concerning the last assertion of the Lemma,
recalling again Lemma~\ref{lemmaT0+} and taking into account the definition of $\rho_\eps^A(t)$,
if $\|A\|_{C^2}<\delta$ for $\delta>0$ small enough, we conclude the proof.
\end{proof}

\medskip

\section{Proof of Theorem~\ref{th:main} completed}
\label{proof-section}

\subsection{First conclusion of Theorem~\ref{th:main}}
Let $T_0$ be as in Lemma~\ref{lemmaT0} and $\eps_0,\delta$ as in Lemma~\ref{lemmaT0+}. By Lemma~\ref{OOmega} and
the definition~\eqref{Tepsdef} it follows that $T^*_\eps=T_0/\eps$, for all $\eps\in
(0, \eps_0)$. Hence, $\Omega_\eps(t)\leq C\eps^2$
for all $\eps\in (0,\eps_0)$ and $t\in [0,T_0/\eps]$, in light of Lemma~\ref{OOmega}.
Moreover, by Theorem~\ref{chiave} there exist functions
$\theta^1_{\eps}, \dots, \theta^m_{\eps}:\R^+\to[0,2\pi)$ and $y_{\eps}:\R^+\to\R^N$ such that
$$
\phi_\eps^j(x,t)=e^{\iu (\xi_\eps(t)\cdot
x+\theta^j_\eps(t)+A(\eps x_\eps(t)) \cdot
(x-x_\eps(t))}r_j(x-y_\eps(t))+\omega^j_\eps(t),
$$
where $\|\omega^j_\eps(t)\|_{\H_{\eps}}\leq C\sqrt{\Omega_\eps(t)}+{\mathcal O}(\eps)$,
and hence, we have
$\|\omega^j_\eps(t)\|_{\H_{\eps}}\leq {\mathcal O}(\eps)$,
for all $\eps\in(0,\eps_{0})$, $t\in [0,T_0/\eps]$ and $j=1, \dots, m$.
Lemma~\ref{lemmaT0} and Lemma~\ref{OOmega} also yield
$|x_\eps(t)-y_{\eps}(t)|\leq {\mathcal O}(\eps)$,
for all $\eps\in(0,\eps_{0})$ and $t\in [0,T_0/\eps]$.
Finally, using {\bf (A)}, {\bf (V)} and~\eqref{Hamilt}, we get
\begin{equation}\label{3}
\begin{aligned}
\Big\|& e^{\iu(\xi_\eps(t)\cdot x+\theta^j_\eps(t)+A(\eps x_\eps(t))
\cdot (x-x_\eps(t))}\left(r_j(x-y_\eps(t))-
r_j(x-x_\eps(t))\right)\Big\|_{H^1}^2\\
& \leq \int\Big|\xi_\eps(t)+A(\eps x_\eps (t))\Big|^2\Big|r_j(x-y_\eps(t))-
r_j(x-x_\eps(t))\Big|^2dx\\
& +\int\Big|\nabla
r_j(x-y_\eps(t))-
\nabla r_j(x-x_\eps(t))\Big|^2 dx\\
& +\int\Big|r_j(x-y_\eps(t))-
r_j(x-x_\eps(t))\Big|^2dx
\leq C|x_\eps(t)-y_{\eps}(t)|^2\leq C{\mathcal O}(\eps^2),
\end{aligned}
\end{equation}
for all $\eps\in(0,\eps_{0})$, $t\in [0,T_0/\eps]$.
Therefore, it follows that
\begin{equation}
\Big\|\phi_\eps^j(x,t) -e^{\iu(\xi_\eps(t)\cdot
x+\theta^j_\eps(t)+A(\eps x_\eps(t)) \cdot
(x-x_\eps(t))}r_j(x-x_\eps(t))\Big\|_{H^1}^2\leq
\mathcal O(\eps^2),
\end{equation}
for all $\eps\in(0,\eps_{0})$, $t\in [0,T_0/\eps]$ and $j=1, \dots, m$.
Hence, Theorem~\ref{th:main} holds true in $[0,T_0/\eps]$.
Now, let us take $x_1^\eps=x_\eps(T_0/\eps)$ and $\xi_1=\xi_\eps(T_0/\eps)$ as new initial
datum in system~\eqref{DriveS} and the functions
$$
\phi_1^j(x)=r_j(x-x_1^\eps)
e^{\iu[A(\eps x_1^\eps)\cdot(x-x_1^\eps)+x\cdot\xi_1^\eps]},\quad
x\in\R^N,\,\, j=1,\dots,m
$$ as new initial datum for problem~\eqref{problema}. Arguing as above, we can show that
Theorem~\ref{th:main} holds true in $[T_0/\eps, 2T_0/\eps]$ and so, in any
finite time interval $[0,T/\eps]$, with $T>0$. The proof of Theorem~\ref{th:main} is now complete
under the assumption that $\|A\|_{C^2}<\delta$.

\subsection{Second conclusion of Theorem~\ref{th:main}}
To prove the second part of Theorem~\ref{th:main}, namely formula~\eqref{secondconcl-A},
we follow the argument of~\cite{selvit} (which is based upon the original
paper by Bronski-Jerrard~\cite{bronski}). Let us give a brief sketch of the proof.
Based upon the identity (see for instance~\cite[p.2571]{selvit}) holding for all $v\in H^1(\R^N)$
$$
\Big|\frac{\nabla v}{\iu}-A(\eps x)v\Big|^2=\frac{|p^{A(\eps x)}(v)|^2}{|v|^2}+|\nabla |v||^2,
\qquad
p^A(v):=\im\big(\bar v(\nabla v(x,t)-\iu A(\eps x)v(x,t))\big),
$$
the energy functional of the Schr\"odinger problem is rewritten as
$$
E_\eps(t)=E_\eps^{{\rm pot}}(t)+E_\eps^{{\rm b}}(t)+E_\eps^{{\rm k}}(t)+E_\eps^{{\rm nl}}(t),
$$
where we have set
\begin{align*}
E_\eps^{{\rm pot}}(t)&:=	\int V(\eps x)|\phi_\eps(x,t)|^2dx,    \\
E_\eps^{{\rm b}}(t) &:=\frac{1}{2}\sum_{j=1}^m\int 	|\nabla |\phi_\eps^j|(x,t)|^2
-\frac{1}{p+1}\sum_{j=1}^m\alpha_j\int |\phi_\eps^j(x,t)|^{2p+2}dx \\
& -\frac{1}{p+1}\sum_{i,j,\,i\neq j}^m \gamma_{ij}\int |\phi_\eps^i(x,t)|^{p+1}|\phi_\eps^j(x,t)|^{p+1}dx , \\
E_\eps^{{\rm k}}(t) &:=	\frac{1}{2}\sum_{j=1}^m\int \frac{|(p^{A(\eps x)}(x,t))^j|^2}{|\phi_\eps^j(x,t)|^2}dx, \\
E_\eps^{{\rm nl}}(t) &:=-\frac{1}{2}\sum_{j=1}^m\beta_j \iint \Phi(\eps(x-y))|\phi_\eps^j(x,t)|^2|\phi_\eps^j(y,t)|^2 dx dy \\
 &-\frac{1}{2}\sum_{i,j,\,i\neq j}^m\omega_{ij} \iint \Phi(\eps(x-y))|\phi_\eps^i(x,t)|^2|\phi_\eps^j(y,t)|^2 dx dy.	
\end{align*}
Notice that, with respect to our notations, we have $E_\eps^{{\rm b}}(r_1,\dots,r_m)={\mathcal E}(r_1,\dots,r_m)$
since $r_i$ are real valued and positive functions. Moreover
$E_\eps^{{\rm b}}(|\psi_\eps^1|,\dots,|\psi_\eps^m|)=E_\eps^{{\rm b}}(|\phi_\eps^1|,\dots,|\phi_\eps^m|)={\mathcal E}(|\phi_\eps^1|,\dots,|\phi_\eps^m|)$.
At this stage, keeping in mind that we possess Lemma~\ref{estUNO}, which expands the energy $E_\eps(t)$
up to an error ${\mathcal O}(\eps^2)$, by repeating the steps of the proof of~\cite[Lemma 3.5]{selvit},
it is readily seen that, as $\eps$ goes to zero,
\begin{equation*}
%	\label{estcruc0}
0\leq E_\eps^{{\rm b}}(|\phi_\eps^1|,\dots,|\phi_\eps^m|)-E_\eps^{{\rm b}}(r_1,\dots,r_m)\leq C\hat\Omega_\eps(t)+{\mathcal O}(\eps^2).
\end{equation*}
This conclusion plays the role of Lemma~\ref{lemmaOmegaepsilon} and, as a consequence,
by the non-degeneracy/energy convexity property (applied with $U=(|\phi_\eps^1|,\dots,|\phi_\eps^m|)$,
see e.g.~\cite[Proposition 1]{bronski} for the scalar case), yields
\begin{equation}
	\label{estcruc0}
\|(|\phi_\eps^1|,\dots,|\phi_\eps^m|)-(r_1(\cdot +y_\eps(t)), \dots,r_m(\cdot +y_\eps(t))\big) \|_{H^1}^{2}
\leq C\hat\Omega_\eps(t)+{\mathcal O}(\eps^2),
\end{equation}
for some $y_\eps(t)\in\R^N$.

\noindent
Moreover, again by the steps of the proof of~\cite[Lemma 3.5]{selvit}, we get
\begin{equation}
	\label{estcruc}
0\leq E_\eps^{{\rm k}}(t)-\frac{1}{2}\sum_{j=1}^m \frac{\Big|\int (p^{A(\eps x)}(x,t))^j \Big|^2}{m_j}
\leq C\hat\Omega_\eps(t)+{\mathcal O}(\eps^2),
\end{equation}
as $\eps$ goes to zero. To achieve this conclusion, one also needs to take into account
the following elementary inequality (following
from the standard Cauchy-Schwarz inequality)
$$
\Big| \int q^A_\eps(x,t)dx  \Big|^2\leq M\sum_{j=1}^m\frac{\Big|\int (p^{A(\eps x)}(x,t))^j dx  \Big|^2}{m_j},
\qquad t\in\R^+.
$$
Furthermore, for any $j=1,\dots,m$ we have the inequality (see~\cite[inequality below
formula (28)]{selvit}; see also~\cite[formula (3.2)]{bronski})
$$
\frac{1}{2} \int \Big |\frac{(p^{A(\eps x)}(x,t))^j}{|\phi_\eps^j(x)|}
-\frac{\big(\int (p^{A(\eps x)}(x,t))^j\big)}{m_j} |\phi_\eps^j(x)|\Big|^2dx\leq
\frac{1}{2}\int \frac{|(p^{A(\eps x)}(x,t))^j|^2}{|\phi_\eps^j(x)|^2}dx
-\frac{1}{2}\frac{\big|\int (p^{A(\eps x)}(x,t))^j \big|^2}{m_j}.
$$
Summing over $j=1,\dots,m$, we get
$$
\frac{1}{2}\sum_{j=1}^m  \int \Big |\frac{(p^{A(\eps x)}(x,t))^j}{|\phi_\eps^j(x)|}
-\frac{\big(\int (p^{A(\eps x)}(x,t))^j\big)}{m_j} |\phi_\eps^j(x)|\Big|^2dx\leq
E_\eps^{{\rm k}}(t) -\frac{1}{2}\sum_{j=1}^m \frac{\big|\int (p^{A(\eps x)}(x,t))^j \big|^2}{m_j}.
$$
In turn, in light of~\eqref{estcruc}, we obtain
\begin{equation}
	\label{estcruc1}
\int \Big |\frac{(p^{A(\eps x)}(x,t))^j}{|\phi_\eps^j(x)|}
-\frac{\big(\int (p^{A(\eps x)}(x,t))^j\big)}{m_j} |\phi_\eps^j(x)|\Big|^2dx
\leq C\hat\Omega_\eps(t)+{\mathcal O}(\eps^2),
\end{equation}
as $\eps$ goes to zero, for any $j=1,\dots,m$. Inequalities~\eqref{estcruc0} and~\eqref{estcruc1} are
precisely what is needed in order to prove~\eqref{secDDIn} of Lemma~\ref{lemmaT0+} (see the proof of Lemma 6.1 in~\cite{squa},
in particular formula~(6.5) therein; see also the proof of Lemma 4.3 in~\cite{squamont}).
Once inequality~\eqref{secDDIn} of Lemma~\ref{lemmaT0+} holds true the rest of the proof
continues as before, yielding the assertion from inequality~\eqref{estcruc0}.

\bigskip
\bigskip

\end{document}